\documentclass{article}
\usepackage{amsmath,amssymb}
\usepackage{enumerate}

\newtheorem{theorem}{Theorem}[section]

\newtheorem{proposition}[theorem]{Proposition}

\newtheorem{remark}[theorem]{Remark}

\newcommand{\Z}{\mathbb{Z}}

\renewcommand{\ker}{\operatorname{Ker}}
\newcommand{\id}{\operatorname{id}}
\newcommand{\Id}{\operatorname{Id}}

\newcommand{\aut}{\operatorname{Aut}}

\newcommand{\ord}[1]{\vert #1\vert}

\newcommand{\soc}{\operatorname{Soc}}

\newenvironment{proof}{\par\noindent{\bf Proof.}}{$\qed$\par\bigskip}
\newenvironment{modproof}{\par\noindent{\bf Proof.}}{\par\bigskip}
\newcommand{\qed}{\enspace\vrule  height6pt  width4pt  depth2pt}

\usepackage[left=2.5cm,top=2.5cm,right=1.5cm,bottom=2.5cm]{geometry}

\begin{document}

\title{Classification of braces of order $p^3$}
\author{David Bachiller}
\date{}
\maketitle

\begin{abstract}
A classification up to isomorphism of all left braces of order $p^3$, where $p$ is any prime number,  
is given. To this end, we first classify all the left braces of order
$p$ and $p^2$, and then we construct explicitly the hypothesis required in
\cite[Corollary~D]{NirBenDavid}
to build multiplications of left braces.
\end{abstract}



\section{Introduction}
Braces are algebraic structures introduced by Rump in \cite{Rump1} in connection with his work on
the set-theoretic solutions of the Yang-Baxter equation.
A \emph{left brace} is a set $B$ with 
two operations, $+$ and $\cdot$, such that $(B,+)$ 
is an abelian group, $(B,\cdot)$ is a group, and these two operations are related by what is called the 
brace property: $$a\cdot (b+c)+a=a\cdot b+a\cdot c\text{, for every }a,b,c\in B.$$ Right braces are defined 
analogously, but replacing the last property by $(b+c)\cdot a+a=b\cdot a+c\cdot a$. A left and right brace is simply called
a brace.

The main motivation to study this structure is its relation with a particular type of solutions of the Yang-Baxter equation, 
the non-degenerate involutive set-theoretic solutions, in the sense of \cite[Section~3]{Cedo}.
To see this, given a left brace $B$, define $\lambda_a(b):=a\cdot b-a$ for every $a,b\in B$. 
It is not difficult to prove that for all $a\in B$, $\lambda_a$ is an automorphism of the additive group of $B$, and 
that the map $\lambda:(B,\cdot)\to\aut(B,+)$, $a\mapsto\lambda_a$, is a morphism of groups.
Then, by \cite[Lemma~4.1(iii)]{Cedo}, the map $s$ defined by
$$
\begin{array}{cccc}
s:&	B\times B&	\longrightarrow&	B\times B\\
		&	(a,b)&	\longmapsto&	\left(\lambda_a(b),~\lambda^{-1}_{\lambda_a(b)}(a)\right),
\end{array}
$$
is a non-degenerate involutive set-theoretic solution of the Yang-Baxter equation;
this $s$ is called the \emph{associated solution} to the left brace $B$. So, in fact, for any brace structure that we can determine,
we are also computing a solution of the Yang-Baxter equation. This is one of the fundamental reasons for 
desiring a classification of left braces, but there are other results relating this structure and the 
Yang-Baxter equation in a fundamental way; 
see \cite{Cedo} for a good introduction to braces and their relation to the 
Yang-Baxter equation. 

Another important fact about braces is their relation with other algebraic structures.
For instance, there is a bijective correspondence between left braces and groups with a bijective 1-cocycle with respect
to a left action (see \cite[Remark~2]{Rump1}). There is also a bijective correspondence between two-sided braces and radical rings
(see \cite[page~159]{Rump1}).

For all these reasons, a classification of all the left braces of finite order is wanted. 
It is known that the multiplicative group of any finite left
brace is solvable. Then, it is natural to begin studying the case of multiplicative group equal
to a $p$-group.
The only previous effort in this direction is \cite{Rump}, where a complete classification of left braces with 
additive group isomorphic to $\Z/(p^n)$, where $p$ is any prime and $n$ is any positive integer, is accomplished.

The aim of this paper is to give a complete classification of all the left braces of order $p^3$, for any prime $p$. 
We divide this problem in two parts, depending on the socle, and in each of them we use different techniques. 
 The \emph{socle} of a left brace $B$, $\soc(B)$, is defined by
$$
\soc(B):=\{a\in B\mid \lambda_a=\id\}.
$$
In other words, it is the kernel of the morphism of groups $\lambda: (B,\cdot)\to\aut(B,+)$, $a\mapsto \lambda_a$.
On one side, we consider braces with non-trivial socle. In this case, we first classify the brace $B/\soc(B)$, which is
of less order than $B$, and then we use the result \cite[Corollary~D]{NirBenDavid}, that gives all the possible ways to extend 
the brace structure of $B/\soc(B)$ to that of $B$. On the other side, we consider braces with $\soc(B)=\{0\}$. In this case,
$\lambda: (B,\cdot)\to\aut(B,+)$ is a monomorphism, and $(B,\cdot)$ is isomorphic to a subgroup $M\leq\aut(B,+)$. 
Thus a brace structure is determined by 
a bijective map $\pi:M\to (B,+)$ such that $\pi(a\cdot b)=\pi(a)+a(\pi(b))$ for all $a,b\in M\leq\aut(B,+).$

The article is organized as follows. In section 2, we present and prove the preliminar results that 
we will need later on. Specifically, we prove a version of the result \cite[Corollary~D]{NirBenDavid} in 
terms of braces, and classify braces of order $p$ and $p^2$. 
Then, we state the main theorem of the paper,
which consists of a complete list of all the brace structures of order $p^3$ up to isomorphism. The final 
sections are completely devoted to the proof of this result.

\section{Preliminar results}
First we need the following theorem, based on the work of Ben David \cite{NirBenDavid}. It is a reformulation of 
\cite[Corollary~D]{NirBenDavid} in terms of braces. This theorem reduces the classification of braces of a given 
finite order to the classification of braces of smaller order plus finding two morphisms $h$ and $\sigma$ with some 
properties described in the hypothesis of the theorem.

\begin{theorem}\label{BenDavid}
Let $H$ be an abelian group and $B$ be a left brace. Let $\sigma:(B,\cdot)\to \aut(H,+)$ be an injective morphism, 
and $h:(H,+)\to (B,+)$ be a surjective morphism. Suppose that they satisfy $h(\sigma(g)(m))=\lambda_g(h(m))$ 
for all $g\in B$ and $m\in H$. Then, the multiplication over $H$ given by $$x\cdot y:=x+\sigma(h(x))(y)~~\forall x,y\in H,$$ 
defines a structure of left brace on $H$ such that $h$ is a morphism of left braces, 
$\soc(H)=\ker(h)$ and $H/\soc(H)\cong B$ as left braces.

Two of these structures, determined by $\sigma$, $h$ and $\sigma'$, $h'$ respectively, are isomorphic if
and only if there exists an $F\in \aut(H,+)$ such that
$$
\sigma'(h'(m))=F^{-1}\circ\sigma(h(F(m)))\circ F,
$$
for all $m\in H$.

Conversely, suppose that $G$ is a left brace. Then, the map $\sigma:(G/\soc(G),\cdot)\to\aut(G,+)$ induced by
the map $\lambda: (G,\cdot)\to\aut(G,+)$,
and the natural map $h:G\to G/\soc(G)$ satisfy the above properties.
\end{theorem}
\begin{proof}
We have to check that $(H,\cdot)$ is a group and
satisfies the left brace property. First of all, for the associativity, if $u,v,w\in H$,
$$
\begin{array}{rl}
(u\cdot v)\cdot w&=u\cdot v+\sigma(h(u\cdot v))(w)\\
&=u\cdot v+\sigma(h(u+\sigma(h(u))(v)))(w)\\
&=u\cdot v+\sigma(h(u)+h(\sigma(h(u))(v)))(w)\\
&=u\cdot v+\sigma(h(u)+\lambda_{h(u)}(h(v)))(w)\\
&=u\cdot v+\sigma(h(u)\cdot h(v))(w)\\
&=u\cdot v+\left(\sigma(h(u))\circ \sigma(h(v))\right)(w)\\
&=u+\sigma(h(u))(v)+\left(\sigma(h(u))\circ \sigma(h(v))\right)(w)\\
&=u+\sigma(h(u))(v+\sigma(h(v))(w))\\
&=u+\sigma(h(u))(v\cdot w)=u\cdot (v\cdot w).
\end{array}
$$

Next, it is easy to check that $u\cdot 0=0\cdot u=u$, so 0 is the multiplicative neutral element. 
To check that all the elements have an inverse, given $u\in H$, consider $-\sigma(h(u)^{-1})(u)$; 
this element is the inverse of $u$ because
$$
\begin{array}{rl}
[-\sigma(h(u)^{-1})(u)]\cdot u&=-\sigma(h(u)^{-1})(u)+\sigma(h(-\sigma(h(u)^{-1})(u)))(u)\\
&=-\sigma(h(u)^{-1})(u)+\sigma(-h(\sigma(h(u)^{-1})(u)))(u)\\
&=-\sigma(h(u)^{-1})(u)+\sigma(-\lambda_{h(u)^{-1}}(h(u)))(u)\\
&=-\sigma(h(u)^{-1})(u)+\sigma(-h(u)^{-1}h(u)+h(u)^{-1})(u)\\
&=0.
\end{array}
$$
A similar computation shows that $u\cdot [-\sigma(h(u)^{-1})(u)]=0$. Finally, we prove the brace property:
$$
\begin{array}{rl}
u(v+w)+u=&u+\sigma(h(u))(v+w)+u=u+\sigma(h(u))(v)+\sigma(h(u))(w)+u\\
=&uv+uw,
\end{array}
$$
for all $u,v,w\in H$.

We use now that $\sigma$  is injective to determine the socle.
In the brace $(H,+,\cdot)$ that we have just defined, the lambda
maps $\lambda_u$, $u\in H$, coincide with the maps $\sigma(h(u))$
because $\lambda_u(v)=u\cdot
v-u=u+\sigma(h(u))(v)-u=\sigma(h(u))(v)$. Then, 
$\lambda_u=\id$ if and only if $\sigma(h(u))=\id$, which is equivalent to $h(u)=0$ by the injectivity of $\sigma$.
Then, $\soc(H)=\ker(h)$. Finally, note that 
$h(u\cdot v)=h(u+\sigma(h(u))(v))=h(u)+h(\sigma(h(u))(v))=h(u)+\lambda_{h(u)}(h(v))=h(u)\cdot h(v)$
for all $u,v\in H$. Hence $h$ is a morphism of left braces.
Thus, the surjectivity of $h$ implies $H/\soc(H)\cong B$ as left braces.

Let $\sigma':(B,\cdot)\to\aut(H,+)$ be an injective morphism, and $h':(H,+)\to (B,+)$ be a surjective
morphism such that $h'(\sigma'(g)(m))=\lambda_g(h'(m))$, for all $g\in B$ and $m\in H$. Define
$$
x\odot y:=x+\sigma'(h'(x))(y),~\forall x,y\in H.
$$
Suppose that the left braces $(H,+,\cdot)$ and $(H,+,\odot)$ are isomorphic. Then there exists 
an $F\in\aut(H,+)$ such that $F(x\odot y)=F(x)\cdot F(y)$, for all $x,y\in H$. Hence
$$
F(x+\sigma'(h'(x))(y))=F(x)+\sigma(h(F(x)))(F(y)),
$$
that is
$$
F(x)+F(\sigma'(h'(x))(y))=F(x)+\sigma(h(F(x)))(F(y)).
$$
Therefore
$$
\sigma'(h'(x))(y)=F^{-1}(\sigma(h(F(x)))(F(y))=(F^{-1}\circ\sigma(h(F(x)))\circ F)(y).
$$
Thus $\sigma'(h'(x))=F^{-1}\circ\sigma(h(F(x)))\circ F$ for all $x\in H$. Conversely, 
any automorphism $F$ of $(H,+)$ such that
$$
\sigma'(h'(x))=F^{-1}\circ\sigma(h(F(x)))\circ F,\forall x\in H
$$
is an isomorphism of left braces $F:(H,+,\odot)\to (H,+,\cdot)$.

The last part of the result is easy to check.
\end{proof}

\begin{remark}
There are two special uses of the isomorphism condition 
$\sigma'(h'(m))=F^{-1}\circ\sigma(h(F(m)))\circ F$
that will be useful later to simplify some cases.
One way to use it is to change the representation $\sigma$ by a conjugate representation, 
taking into account that $h$ changes to $h'=h\circ F$. So conjugate representations give rise to 
isomorphic braces, with the appropriate change of $h$.

Another way to use this condition is to find an $F$ that commutes with $\sigma(g)$ for all $g$. Then, we can use this $F$
to modify $h$ keeping the same $\sigma$.
\end{remark}

\newcounter{nx}

Our aim is to classify all braces of order $p^3$. If we want to apply the last theorem, we need 
the classification of braces of order $p$ and $p^2$; this is done in the next proposition. We also have
to know the structure of their multiplicative group.

\begin{remark}
Throughout the paper, multiplication without a dot denotes the usual ring multiplication 
over $\Z/(p^n)$, or the usual multiplication of matrices. Dots are always used to denote left brace multiplications.
\end{remark}

\begin{proposition}\label{bracesp2}
A complete list of braces $G$ of order $p$ and $p^2$ up to isomorphism, classified with respect to their 
additive groups, is the following:
\begin{itemize}
 \item Additive group isomorphic to $\Z/(p)$:
 \begin{enumerate}[(i)]
 \item $$x_1\cdot x_2:=x_1+x_2,~\text{ and then }~(G,\cdot)\cong \Z/(p).$$
 \setcounter{nx}{\value{enumi}}
 \end{enumerate} 
 \item Additive group isomorphic to $\Z/(p^2)$:
 \begin{enumerate}[(i)] 
 \setcounter{enumi}{\value{nx}}
 \item $$x_1\cdot x_2:=x_1+x_2,~\text{ and then }~(G,\cdot)\cong \Z/(p^2);$$
 \item $$x_1\cdot x_2:=x_1+x_2+p x_1 x_2,\text{ and then }$$
 $$(G,\cdot)\cong\left\{
 \begin{array}{cl}
 \Z/(2)\times\Z/(2),&~p=2\\
 \Z/(p^2),&~p\neq 2
 \end{array}
 \right.$$
 \setcounter{nx}{\value{enumi}}
 \end{enumerate}
 
 \item Additive group isomorphic to $\Z/(p)\times\Z/(p)$:
 \begin{enumerate}[(i)]
 \setcounter{enumi}{\value{nx}} 
 \item $$
 \begin{pmatrix}
 x_1\\
 y_1\\
 \end{pmatrix}
 \cdot 
 \begin{pmatrix}
 x_2\\
 y_2\\
 \end{pmatrix}:=
 \begin{pmatrix}
 x_1+x_2\\
 y_1+y_2\\
 \end{pmatrix},~\text{ and then }~(G,\cdot)\cong \Z/(p)\times \Z(p);
 $$
 \item $$\begin{pmatrix}x_1\\y_1\end{pmatrix}\cdot \begin{pmatrix}x_2\\y_2\end{pmatrix}:=
 \begin{pmatrix}x_1+x_2+y_1 y_2\\y_1+y_2\end{pmatrix},\text{ and then }$$
 $$(G,\cdot)\cong\left\{
 \begin{array}{cl}
 \Z/(4),&~p=2\\
 \Z/(p)\times\Z/(p),&~p\neq 2
 \end{array}
 \right.$$
 \end{enumerate}
\end{itemize}

\end{proposition}

\begin{modproof}
 \begin{itemize}
\item Brace of order $p$.

For order equal to $p$, we must have additive group and multiplicative 
group both isomorphic to $\Z/(p)$. The only possible morphism
$\lambda:\Z/(p)\to \aut(\Z/(p))\cong (\Z/(p))^*$ is the trivial one, so $x\cdot y:=x+y$ is 
the only possible brace structure.

\item $(G,+)$ isomorphic to $\Z/(p^2)$.

See \cite[Theorem~1 and Proposition~4]{Rump} for the details. There are two possibilities: the trivial pro\-duct 
$x\cdot y:=x+y$, and the multiplication $x\cdot y:=x+(1+p)^xy=x+y+pxy$.
When $p\neq 2$, we have in both cases that $(G,\cdot)\cong \Z/(p^2)$ because 1 is an element of order $p^2$
of the multiplicative groups of this brace. 
When $p=2$, we have $(G,\cdot)\cong \Z/(4)$ in the first case, and $(G,\cdot)\cong \Z/(2)\times\Z/(2)$ 
in the second case because all the elements have multiplicative order equal to  $2$ or $1$.

\item $(G,+)$ isomorphic to $\Z/(p)\times\Z/(p)$.

\paragraph{Socle of order $p^2$.}
Since $\soc(G)=G$, the brace must be trivial; i.e. the multiplication and the sum coincide: $$
\begin{pmatrix}
x_1\\
y_1
\end{pmatrix}\cdot
\begin{pmatrix}
x_2\\
y_2
\end{pmatrix}:=
\begin{pmatrix}
x_1+x_2\\
y_1+y_2
\end{pmatrix}.$$

\paragraph{Socle of order $p$.}
Consider any monomorphism 
$$\sigma:(G/\soc(G),\cdot)\cong\Z/(p)\to \aut(G,+)\cong\aut(\Z/(p)\times\Z/(p))\cong GL_2(\Z/(p)).$$ 
It is determined by the image of 1, which is a matrix 
$A\in GL_2(\Z/(p))$ of order $p$. Consider also a surjective morphism
$$h:(G,+)\cong\Z/(p)\times\Z/(p)\to (G/\soc(G),+)\cong\Z/(p),$$ which is determined by a non-zero matrix $(\alpha,\beta)$,
with $\alpha,\beta\in\Z/(p)$,
$$h(x,y)=\alpha x+\beta y=
\begin{pmatrix}
\alpha & \beta \\
\end{pmatrix}
\begin{pmatrix}
x \\
y \\
\end{pmatrix}.$$
The condition $h(\sigma(k)(x,y))=\lambda_k(h(x,y))=h(x,y)$ is equivalent in this case 
to $(\alpha,\beta)A=(\alpha,\beta)$, so $(\alpha,\beta)$ must be an eigenvector
of eigenvalue 1 of $A^t$. Then, by Theorem \ref{BenDavid}, any structure of left brace on $\Z/(p)\times\Z/(p)$ 
in this case is given by 
$$
\begin{pmatrix}
x_1\\
y_1
\end{pmatrix}\cdot
\begin{pmatrix}
x_2\\
y_2
\end{pmatrix}:=
\begin{pmatrix}
x_1\\
y_1
\end{pmatrix}+\sigma(h(x_1,y_1))\begin{pmatrix}x_2\\y_2\end{pmatrix}=
\begin{pmatrix}
x_1\\
y_1
\end{pmatrix}
+A^{h(x_1,y_1)} 
\begin{pmatrix}
x_2\\
y_2
\end{pmatrix}.$$

We have to find which of this braces are isomorphic. 
The condition of isomorphism is in this case
$$
A^{(\alpha,~\beta) (x,y)^t}=F^{-1} A'^{(\alpha',~\beta') F(x,y)^t} F=
\left(F^{-1} A' F\right)^{(\alpha',~\beta') F(x,y)^t},
$$
for some $F\in GL_2(\Z/(p))$.

One possible way to use this condition is to change the matrix $A$ by one of its conjugates, 
taking into account that the vector $(\alpha,\beta)$ is multiplied by $F$ on the right.
Since $A$ has order $p$, $0=A^p-\Id=(A-\Id)^p$, and thus its minimal polynomial divides $(x-1)^p$. 
Then $A$ is conjugate to a matrix of Jordan form of eigenvalue 1, 
so we may take 
$A=\begin{pmatrix}
1 & 1 \\
0 & 1 \\
\end{pmatrix}.$ 
Since $(\alpha,\beta)$ must 
be an eigenvector of $A^t$, we have $(\alpha,\beta)=(0,\gamma)$, $\gamma\neq 0$. But using 
$F=\gamma^{-1}\Id,$ we obtain 
$A=\begin{pmatrix}
1 & 1 \\
0 & 1 \\
\end{pmatrix}$ and $(\alpha,\beta)=(0,1)$. So the only structure of brace up to isomorphism
is 
\begin{eqnarray*}
\begin{pmatrix}
 x_1\\
 y_1
\end{pmatrix}
\cdot 
\begin{pmatrix}
 x_2\\
 y_2
\end{pmatrix}
&:=&\begin{pmatrix}
   x_1\\
   y_1
  \end{pmatrix}
+\begin{pmatrix}
1 & 1 \\
0 & 1 \\
\end{pmatrix}^{\begin{pmatrix}
0&1
\end{pmatrix}\begin{pmatrix}
x_1 \\
y_1 \\
\end{pmatrix}}\begin{pmatrix}
x_2 \\
y_2 \\
\end{pmatrix}\\[5pt]
&=&\begin{pmatrix}
  x_1\\
  y_1
 \end{pmatrix}
+\begin{pmatrix}
1 & y_1 \\
0 & 1 \\
\end{pmatrix}
\begin{pmatrix}
x_2 \\
y_2 \\
\end{pmatrix}=
\begin{pmatrix}
 x_1+x_2+y_1y_2\\
 y_1+y_2
\end{pmatrix}.
\end{eqnarray*}

We can compute all the powers of any element of a left brace using the formula
\begin{equation}\label{powers}
 x^n=x+\lambda_x(x)+\lambda^2_x(x)+\cdots +\lambda_x^{n-1}(x)=(\id+\lambda_x+\lambda^2_x+\cdots+\lambda_x^{n-1})(x),\tag{$\star$}
\end{equation}
which is easy to prove by induction. In this paper, the $\lambda_x$'s are always matrices, 
and then, to be able to apply this formula, we only need to compute
powers of matrices by induction, and then add all of them.

Specifically, in our present case, 
$\lambda_{(x,y)}=
\begin{pmatrix}
1 & y \\
0 & 1 \\
\end{pmatrix}$ and 
$\lambda^n_{(x,y)}=
\begin{pmatrix}
1 & ny \\
0 & 1 \\
\end{pmatrix}$, so
$$
\begin{pmatrix}
 x\\
 y
\end{pmatrix}^n=
\left(
\begin{pmatrix}
 1& 0\\
 0& 1\\
\end{pmatrix}+
\begin{pmatrix}
 1& y\\
 0& 1\\
\end{pmatrix}+
\begin{pmatrix}
 1& 2y\\
 0& 1\\
\end{pmatrix}+\cdots
+\begin{pmatrix}
 1& (n-1)y\\
 0& 1\\
\end{pmatrix}
\right)
\begin{pmatrix}
 x\\
 y
\end{pmatrix}
=
\begin{pmatrix}
 nx+y^2\displaystyle\sum^{n-1}_{i=1} i\\
 ny
\end{pmatrix}.
$$

When $p\neq 2$, $(G,\cdot)$ has exponent $p$ because $(x,y)^p=(px+\frac{p(p-1)}{2}y^2,py)=(0,0)$, and thus $(G,\cdot)\cong\Z/(p)\times\Z/(p)$. 
When $p=2$, $(x,y)^2=(y^2,0)$, so $(0,1)$ has order $4$ in the multiplicative group of this brace, and thus
$(G,\cdot)\cong\Z/(4)$.

\paragraph{Trivial socle.}
It is impossible in this case: there is no injective morphism 
$\lambda:(G,\cdot)\to\aut(\Z/(p)\times\Z/(p))$ because 
$\ord{G}=p^2$
and $\ord{\aut(\Z/(p)\times\Z/(p))}=p(p-1)(p^2-1)$.\qed
\end{itemize}\end{modproof}

Finally, since we want to know the structure of the multiplicative group of each 
left brace of order $p^3$, we need to recall the classification of
non-abelian groups of order $p^3$.
\begin{proposition}
Let $G$ be a non-abelian group of order $p^3$.
\begin{enumerate}[(i)]
\item If $p=2$, then $G$ is isomorphic to the dihedral group
$$
D_4=\left\langle x,y~\middle|~ x^2=y^4=1,~yx=xy^3\right\rangle,
$$ 
or to the quaternion group
$$
Q=\left\langle x,y~\middle|~ x^2=y^2,~x^4=y^4=1,~y^{-1}xy=x^{-1}\right\rangle.
$$
\item If $p\neq 2$, then $G$ is isomorphic to 
$$
M(p)=\left\langle x,y,z~\middle|~ x^p=y^p=z^p=1,~[x,z]=[x,y]=1,~[x,y]=z\right\rangle,
$$
or to
$$
M_3(p)=\left\langle x,y~\middle|~ x^{p^2}=y^p=1,~y^{-1}xy=x^{1+p}\right\rangle.
$$
\end{enumerate}
\end{proposition}
\begin{proof}
See \cite[Theorem~5.4.4]{Gorenstein}.
\end{proof}

Observe that one difference between $M(p)$ and $M_3(p)$ is that the former has exponent $p$, while the later
 has elements of order $p^2$. Note also that $Q$ has only 
one element of order $2$, but $D_4$ has five elements of order $2$. So to differenciate between one group
of order $p^3$ or another,
we only have to determine if the multiplication is abelian or not, and to compute the order of its elements. 
We will skip the details of this last part, since the computations
are similar to the ones done in the proof of Proposition~\ref{bracesp2} using the formula (\ref{powers}).

\section{Main theorems}

\begin{theorem}{(Case $p=2$)}
The following is a complete list of left braces $G$ of order $8$ up to isomorphism,
classified with respect to its additive group, and then with respect to the order of its socle:
\begin{enumerate}
\item Additive group isomorphic to $\Z/(8)$:
\begin{itemize}
\item Socle of order $2$
$$
x_1\cdot x_2:=x_1+x_2+2x_1x_2
$$
$$
(G,\cdot)\cong \Z/(2)\times\Z/(4)
$$

\item Socle of order $4$
$$
x_1\cdot x_2:=x_1+(1+2\alpha)^{x_1}x_2,\text{ for }\alpha=1,2,3
$$
$$(G,\cdot)\cong\left\{
 \begin{array}{cl}
 Q,&~\alpha=1\\
 \Z/(8),&~\alpha= 2\\
 D_4,&~\alpha=3
 \end{array}
 \right.$$

\item Socle of order $8$
$$
x_1\cdot x_2:=x_1+x_2
$$
$$
(G,\cdot)\cong \Z/(8)
$$
\end{itemize}

\item Additive group isomorphic to $\Z/(2)\times\Z/(4)$
\begin{itemize}
\item Socle of order 1:
$$
\begin{pmatrix}y_1\\z_1+2x_1\end{pmatrix}\cdot \begin{pmatrix}y_2\\z_2+2x_2\end{pmatrix}:=
\begin{pmatrix}
y_1+y_2+(x_1+y_1+z_1+y_1z_1)z_2\\
z_1+2x_1+2z_1y_2+2(y_1+x_1z_1)z_2+z_2+2x_2
\end{pmatrix}
$$
$$
(G,\cdot)\cong D_4
$$

\item Socle of order $2$:
$$
\begin{pmatrix}x_1\\y_1\end{pmatrix}\cdot\begin{pmatrix}x_2\\y_2\end{pmatrix}
:=\begin{pmatrix}
x_1+x_2\\
y_1+y_2+2x_1x_2+2y_1y_2
\end{pmatrix}
$$
$$
(G,\cdot)\cong\Z/(2)\times\Z/(4)
$$

$$
\begin{pmatrix}x_1\\y_1\end{pmatrix}\cdot\begin{pmatrix}x_2\\y_2\end{pmatrix}
:=\begin{pmatrix}
x_1+x_2\\
y_1+y_2+2(x_1+y_1)x_2+2y_1y_2
\end{pmatrix}
$$
$$
(G,\cdot)\cong D_4
$$

$$
\begin{pmatrix}x_1\\y_1\end{pmatrix}\cdot\begin{pmatrix}x_2\\y_2\end{pmatrix}
:=\begin{pmatrix}
x_1+x_2\\
y_1+y_2+2y_1x_2+2x_1y_2
\end{pmatrix}
$$
$$
(G,\cdot)\cong\Z/(2)\times\Z/(4)
$$

$$
\begin{pmatrix}x_1\\y_1\end{pmatrix}\cdot\begin{pmatrix}x_2\\y_2\end{pmatrix}
:=\begin{pmatrix}
x_1+x_2\\
y_1+y_2+2y_1x_2+2(x_1+y_1)y_2
\end{pmatrix}
$$
$$
(G,\cdot)\cong (\Z/(2))^3
$$

$$
\begin{pmatrix}x_1\\y_1\end{pmatrix}\cdot\begin{pmatrix}x_2\\y_2\end{pmatrix}
:=\begin{pmatrix}
x_1+x_2+y_1y_2\\
y_1+y_2+2y_1x_2+2x_1y_2
\end{pmatrix}
$$
$$
(G,\cdot)\cong\Z/(2)\times\Z/(4)
$$

$$
\begin{pmatrix}x_1\\y_1\end{pmatrix}\cdot\begin{pmatrix}x_2\\y_2\end{pmatrix}
:=\begin{pmatrix}
x_1+x_2+y_2\displaystyle\sum^{y_1-1}_{i=1} i\\
y_1+y_2+2y_1y_2
\end{pmatrix}
$$
$$
(G,\cdot)\cong D_4
$$

\item Socle of order $4$:
$$
\begin{pmatrix}x_1\\y_1\end{pmatrix}\cdot\begin{pmatrix}x_2\\y_2\end{pmatrix}
:=\begin{pmatrix}
x_1+x_2\\
y_1+y_2+2x_1y_2
\end{pmatrix}
$$
$$
(G,\cdot)\cong D_4
$$

$$
\begin{pmatrix}x_1\\y_1\end{pmatrix}\cdot\begin{pmatrix}x_2\\y_2\end{pmatrix}
:=\begin{pmatrix}
x_1+x_2\\
y_1+y_2+2y_1y_2
\end{pmatrix}
$$
$$
(G,\cdot)\cong (\Z/(2))^3
$$

$$
\begin{pmatrix}x_1\\y_1\end{pmatrix}\cdot\begin{pmatrix}x_2\\y_2\end{pmatrix}
:=\begin{pmatrix}
x_1+x_2\\
y_1+y_2+2y_1x_2
\end{pmatrix}
$$
$$
(G,\cdot)\cong D_4
$$

$$
\begin{pmatrix}x_1\\y_1\end{pmatrix}\cdot\begin{pmatrix}x_2\\y_2\end{pmatrix}
:=\begin{pmatrix}
x_1+x_2\\
y_1+y_2+2x_1x_2
\end{pmatrix}
$$
$$
(G,\cdot)\cong\Z/(2)\times\Z/(4)
$$

$$
\begin{pmatrix}x_1\\y_1\end{pmatrix}\cdot\begin{pmatrix}x_2\\y_2\end{pmatrix}
:=\begin{pmatrix}
x_1+x_2\\
y_1+y_2+2(x_1+y_1)x_2
\end{pmatrix}
$$
$$
(G,\cdot)\cong Q
$$

$$
\begin{pmatrix}x_1\\y_1\end{pmatrix}\cdot\begin{pmatrix}x_2\\y_2\end{pmatrix}
:=\begin{pmatrix}
x_1+x_2+y_1y_2\\
y_1+y_2
\end{pmatrix}
$$
$$
(G,\cdot)\cong \Z/(2)\times\Z/(4)
$$

\item Socle of order $8$:
$$
\begin{pmatrix}x_1\\y_1\end{pmatrix}\cdot\begin{pmatrix}x_2\\y_2\end{pmatrix}
:=\begin{pmatrix}
x_1+x_2\\
y_1+y_2
\end{pmatrix}
$$
$$
(G,\cdot)\cong\Z/(2)\times\Z/(4)
$$
\end{itemize}

\item Additive group isomorphic to $(\Z/(2))^3$
\begin{itemize}
\item Socle of order 1:
$$
\begin{pmatrix}x_1\\y_1\\z_1\end{pmatrix}\cdot\begin{pmatrix}x_2\\y_2\\z_2\end{pmatrix}
:=\begin{pmatrix}
x_1+x_2+z_1y_2+x_1z_2+y_1z_2+x_1z_1z_2\\
y_1+y_2+z_1z_2+x_1z_2+y_1z_1z_2\\
z_1+z_2
\end{pmatrix}
$$
$$
(G,\cdot)\cong D_4
$$

\item Socle of order $2$:
$$
\begin{pmatrix}x_1\\y_1\\z_1\end{pmatrix}\cdot\begin{pmatrix}x_2\\y_2\\z_2\end{pmatrix}
:=\begin{pmatrix}
x_1+x_2+z_1y_2+y_1z_2\\
y_1+y_2\\
z_1+z_2
\end{pmatrix}
$$
$$
(G,\cdot)\cong(\Z/(2))^3
$$

$$
\begin{pmatrix}x_1\\y_1\\z_1\end{pmatrix}\cdot\begin{pmatrix}x_2\\y_2\\z_2\end{pmatrix}
:=\begin{pmatrix}
x_1+x_2+y_1y_2+z_1z_2\\
y_1+y_2\\
z_1+z_2
\end{pmatrix}
$$
$$
(G,\cdot)\cong\Z/(2)\times\Z/(4)
$$

$$
\begin{pmatrix}x_1\\y_1\\z_1\end{pmatrix}\cdot\begin{pmatrix}x_2\\y_2\\z_2\end{pmatrix}
:=\begin{pmatrix}
x_1+x_2+(y_1+z_1)y_2+z_1z_2\\
y_1+y_2\\
z_1+z_2
\end{pmatrix}
$$
$$
(G,\cdot)\cong Q
$$

$$
\begin{pmatrix}x_1\\y_1\\z_1\end{pmatrix}\cdot\begin{pmatrix}x_2\\y_2\\z_2\end{pmatrix}
:=\begin{pmatrix}
x_1+x_2+z_1y_2+y_1z_2\\
y_1+y_2+z_1z_2\\
z_1+z_2
\end{pmatrix}
$$
$$
(G,\cdot)\cong\Z/(2)\times\Z/(4)
$$

\item Socle of order $4$:
$$
\begin{pmatrix}x_1\\y_1\\z_1\end{pmatrix}\cdot\begin{pmatrix}x_2\\y_2\\z_2\end{pmatrix}
:=\begin{pmatrix}
x_1+x_2+y_1y_2\\
y_1+y_2\\
z_1+z_2
\end{pmatrix}
$$
$$
(G,\cdot)\cong\Z/(2)\times\Z/(4)
$$

$$
\begin{pmatrix}x_1\\y_1\\z_1\end{pmatrix}\cdot\begin{pmatrix}x_2\\y_2\\z_2\end{pmatrix}
:=\begin{pmatrix}
x_1+x_2+z_1y_2\\
y_1+y_2\\
z_1+z_2
\end{pmatrix}
$$
$$
(G,\cdot)\cong D_4
$$

\item Socle of order $8$:
$$
\begin{pmatrix}x_1\\y_1\\z_1\end{pmatrix}\cdot\begin{pmatrix}x_2\\y_2\\z_2\end{pmatrix}
:=\begin{pmatrix}
x_1+x_2\\
y_1+y_2\\
z_1+z_2
\end{pmatrix}
$$
$$
(G,\cdot)\cong (\Z/(2))^3
$$
\end{itemize}
\end{enumerate}
\end{theorem}

\paragraph{Notation:} $C(n,m)$ denotes the binomial coefficient, which is equal to $\displaystyle\frac{n!}{m!(n-m)!}$ 
if $n\geq m$, and equal to $0$ if $n<m$.

\begin{theorem}{(Case $p\neq 2$)}
Let $p$ be a prime different from $2$. 
Let $\varepsilon$ be a fixed element of $\Z/(p)$ which is not a square. 
Then, the following is a complete list of left braces $G$ of order $p^3$ up to isomorphism,
 classified with respect to its additive group, and then with respect to the order of its socle:
\begin{enumerate}
\item Additive group isomorphic to $\Z/(p^3)$:
\begin{itemize}
\item Socle of order $p$:
$$
x_1\cdot x_2:=x_1+x_2+px_1x_2
$$
$$
(G,\cdot)\cong\Z/(p^3)
$$

\item Socle of order $p^2$:
$$
x_1\cdot x_2:=x_1+x_2+p^2x_1x_2
$$
$$
(G,\cdot)\cong\Z(p^3)
$$

\item Socle of order $p^3$:
$$
x_1\cdot x_2:=x_1+x_2
$$
$$
(G,\cdot)\cong\Z/(p^3)
$$
\end{itemize}

\item Additive group isomorphic to $\Z/(p)\times\Z/(p^2)$
\begin{itemize}
\item Socle of order $p$:
$$
\begin{pmatrix}x_1\\y_1\end{pmatrix}\cdot\begin{pmatrix}x_2\\y_2\end{pmatrix}
:=\begin{pmatrix}
x_1+x_2\\
y_1+y_2+p(\varepsilon x_1+\lambda y_1)x_2+py_1y_2
\end{pmatrix}\text{, for each }\lambda\in\{0,1,\dots,(p-1)/2\}
$$
$$
(G,\cdot)\cong\left\{
 \begin{array}{cl}
 \Z/(p)\times\Z/(p^2),&~\lambda=0\\
 M_3(p),&~\lambda\neq 0\\
 \end{array}
 \right.
$$

$$
\begin{pmatrix}x_1\\y_1\end{pmatrix}\cdot\begin{pmatrix}x_2\\y_2\end{pmatrix}
:=\begin{pmatrix}
x_1+x_2\\
y_1+y_2+p(x_1+\lambda y_1)x_2+py_1y_2
\end{pmatrix}\text{, for each }\lambda\in\{0,1,\dots,(p-1)/2\}
$$
$$
(G,\cdot)\cong\left\{
 \begin{array}{cl}
 \Z/(p)\times\Z/(p^2),&~\lambda=0\\
 M_3(p),&~\lambda\neq 0\\
 \end{array}
 \right.
$$

$$
\begin{pmatrix}x_1\\y_1\end{pmatrix}\cdot\begin{pmatrix}x_2\\y_2\end{pmatrix}
:=\begin{pmatrix}
x_1+x_2\\
y_1+y_2+p\lambda y_1x_2+px_1y_2
\end{pmatrix}\text{, for each }\lambda\in\{1,2,\dots,p-1\}
$$
$$
(G,\cdot)\cong\left\{
 \begin{array}{cl}
 \Z/(p)\times\Z/(p^2),&~\lambda=1\\
 M_3(p),&~\lambda\neq 1\\
 \end{array}
 \right.
$$

$$
\begin{pmatrix}x_1\\y_1\end{pmatrix}\cdot\begin{pmatrix}x_2\\y_2\end{pmatrix}
:=\begin{pmatrix}
x_1+x_2\\
y_1+y_2-py_1x_2+p(x_1+y_1)y_2
\end{pmatrix}
$$
$$
(G,\cdot)\cong M_3(p)
$$

$$
\begin{pmatrix}x_1\\y_1\end{pmatrix}\cdot\begin{pmatrix}x_2\\y_2\end{pmatrix}
:=\begin{pmatrix}
x_1+x_2+y_1y_2\\
y_1+y_2+pay_1x_2+px_1y_2+p(a-1)~ C(y_1, 2)~y_2
\end{pmatrix}\text{, for each }a\in\{0,1,\dots,p-1\}
$$
$$
p\neq 3,~~(G,\cdot)\cong\left\{
 \begin{array}{cl}
 \Z/(p)\times\Z/(p^2),&~a=1\\
 M_3(p),&~a\neq 1\\
 \end{array}
 \right.
$$
$$
p=3,~~(G,\cdot)\cong\left\{
 \begin{array}{cl}
 \Z/(3)\times\Z/(9),&~a=1\\
 M(3),&~a= -1\\
 M_3(3),&~a=0
 \end{array}
 \right.
$$

$$
\begin{pmatrix}x_1\\y_1\end{pmatrix}\cdot\begin{pmatrix}x_2\\y_2\end{pmatrix}
:=\begin{pmatrix}
x_1+x_2+y_1y_2\\
y_1+y_2+pay_1x_2+p\varepsilon x_1y_2+p(a-\varepsilon)~C(y_1, 2)~y_2
\end{pmatrix}\text{, for each }a\in\{0,1,\dots,p-1\}
$$
$$
p\neq 3,~~(G,\cdot)\cong\left\{
 \begin{array}{cl}
 \Z/(p)\times\Z/(p^2),&~a=\varepsilon\\
 M_3(p),&~a\neq \varepsilon\\
 \end{array}
 \right.
$$
$$
p=3,~~(G,\cdot)\cong\left\{
 \begin{array}{cl}
 M_3(3),&~a=1\\
 (\Z/(3))^3,&~a= -1\\
 M_3(3),&~a=0
 \end{array}
 \right.
$$

\item Socle of order $p^2$:
$$
\begin{pmatrix}x_1\\y_1\end{pmatrix}\cdot\begin{pmatrix}x_2\\y_2\end{pmatrix}
:=\begin{pmatrix}
x_1+x_2\\
y_1+y_2+px_1y_2
\end{pmatrix}
$$
$$
(G,\cdot)\cong M_3(p)
$$

$$
\begin{pmatrix}x_1\\y_1\end{pmatrix}\cdot\begin{pmatrix}x_2\\y_2\end{pmatrix}
:=\begin{pmatrix}
x_1+x_2\\
y_1+y_2+py_1y_2
\end{pmatrix}
$$
$$
(G,\cdot)\cong\Z/(p)\times\Z/(p^2)
$$

$$
\begin{pmatrix}x_1\\y_1\end{pmatrix}\cdot\begin{pmatrix}x_2\\y_2\end{pmatrix}
:=\begin{pmatrix}
x_1+x_2\\
y_1+y_2+py_1x_2
\end{pmatrix}
$$
$$
(G,\cdot)\cong M_3(p)
$$

$$
\begin{pmatrix}x_1\\y_1\end{pmatrix}\cdot\begin{pmatrix}x_2\\y_2\end{pmatrix}
:=\begin{pmatrix}
x_1+x_2\\
y_1+y_2+px_1x_2
\end{pmatrix}
$$
$$
(G,\cdot)\cong\Z/(p)\times\Z/(p^2)
$$

$$
\begin{pmatrix}x_1\\y_1\end{pmatrix}\cdot\begin{pmatrix}x_2\\y_2\end{pmatrix}
:=\begin{pmatrix}
x_1+x_2\\
y_1+y_2+p(x_1+y_1)x_2
\end{pmatrix}
$$
$$
(G,\cdot)\cong M_3(p)
$$

$$
\begin{pmatrix}x_1\\y_1\end{pmatrix}\cdot\begin{pmatrix}x_2\\y_2\end{pmatrix}
:=\begin{pmatrix}
x_1+x_2+y_1y_2\\
y_1+y_2
\end{pmatrix}
$$
$$
(G,\cdot)\cong\Z/(p)\times\Z/(p^2)
$$

$$
\begin{pmatrix}x_1\\y_1\end{pmatrix}\cdot\begin{pmatrix}x_2\\y_2\end{pmatrix}
:=\begin{pmatrix}
x_1+x_2+y_1y_2\\
y_1+y_2+py_1x_2+p~C(y_1, 2)~ y_2
\end{pmatrix}
$$
$$
(G,\cdot)\cong M_3(p)
$$

$$
\begin{pmatrix}x_1\\y_1\end{pmatrix}\cdot\begin{pmatrix}x_2\\y_2\end{pmatrix}
:=\begin{pmatrix}
x_1+x_2+y_1y_2\\
y_1+y_2+p\varepsilon y_1x_2+p\varepsilon~ C(y_1, 2)~y_2
\end{pmatrix}.
$$
$$
(G,\cdot)\cong M_3(p)
$$

\item Socle of order $p^3$:
$$
\begin{pmatrix}x_1\\y_1\end{pmatrix}\cdot\begin{pmatrix}x_2\\y_2\end{pmatrix}
:=\begin{pmatrix}
x_1+x_2\\
y_1+y_2
\end{pmatrix}
$$
$$
(G,\cdot)\cong M_3(p)
$$
\end{itemize}

\item Additive group isomorphic to $(\Z/(p))^3$
\begin{itemize}
\item Socle of order $p$:
$$
\begin{pmatrix}x_1\\y_1\\z_1\end{pmatrix}\cdot\begin{pmatrix}x_2\\y_2\\z_2\end{pmatrix}
:=\begin{pmatrix}
x_1+x_2-z_1y_2+y_1z_2\\
y_1+y_2\\
z_1+z_2
\end{pmatrix}
$$
$$
(G,\cdot)\cong M(p)
$$

$$
\begin{pmatrix}x_1\\y_1\\z_1\end{pmatrix}\cdot\begin{pmatrix}x_2\\y_2\\z_2\end{pmatrix}
:=\begin{pmatrix}
x_1+x_2+(y_1+\lambda z_1)y_2+z_1z_2\\
y_1+y_2\\
z_1+z_2
\end{pmatrix}\text{, for each }\lambda\in\{0,1,\dots ,(p-1)/2\}
$$
$$
(G,\cdot)\cong\left\{
 \begin{array}{cl}
 (\Z/(p))^3,&~\lambda=0\\
 M(p),&~\lambda\neq 0\\
 \end{array}
 \right.
$$

$$
\begin{pmatrix}x_1\\y_1\\z_1\end{pmatrix}\cdot\begin{pmatrix}x_2\\y_2\\z_2\end{pmatrix}
:=\begin{pmatrix}
x_1+x_2+(\varepsilon y_1+\lambda z_1)y_2+ z_1z_2\\
y_1+y_2\\
z_1+z_2
\end{pmatrix}\text{, for each }\lambda\in\{0,1,\dots ,(p-1)/2\}
$$
$$
(G,\cdot)\cong\left\{
 \begin{array}{cl}
 (\Z/(p))^3,&~\lambda=0\\
 M(p),&~\lambda\neq 0\\
 \end{array}
 \right.
$$

$$
\begin{pmatrix}x_1\\y_1\\z_1\end{pmatrix}\cdot\begin{pmatrix}x_2\\y_2\\z_2\end{pmatrix}
:=\begin{pmatrix}
x_1+x_2+cz_1y_2+y_1z_2+(c-1)~C(z_1, 2)~z_2\\
y_1+y_2+z_1z_2\\
z_1+z_2
\end{pmatrix}\text{, for each }c\in\Z/(p).
$$
$$
p=3,~~(G,\cdot)\cong\left\{
 \begin{array}{cl}
 M(3),&~c=0\\
 \Z/(3)\times\Z/(9),&~c=1\\
 M_3(3),&~c\neq 0,1
 \end{array}
 \right.
$$
$$
p\neq 3,~~(G,\cdot)\cong\left\{
 \begin{array}{cl}
 (\Z/(p))^3,&~c= 1\\
 M(p),&~c\neq 1
 \end{array}
 \right.
$$

\item Socle of order $p^2$:
$$
\begin{pmatrix}x_1\\y_1\\z_1\end{pmatrix}\cdot\begin{pmatrix}x_2\\y_2\\z_2\end{pmatrix}
:=\begin{pmatrix}
x_1+x_2+y_1y_2\\
y_1+y_2\\
z_1+z_2
\end{pmatrix}
$$
$$
(G,\cdot)\cong (\Z/(p))^3
$$

$$
\begin{pmatrix}x_1\\y_1\\z_1\end{pmatrix}\cdot\begin{pmatrix}x_2\\y_2\\z_2\end{pmatrix}
:=\begin{pmatrix}
x_1+x_2+z_1y_2\\
y_1+y_2\\
z_1+z_2
\end{pmatrix}
$$
$$
(G,\cdot)\cong M(p)
$$

$$
\begin{pmatrix}x_1\\y_1\\z_1\end{pmatrix}\cdot\begin{pmatrix}x_2\\y_2\\z_2\end{pmatrix}
:=\begin{pmatrix}
x_1+x_2+z_1y_2+C(z_1, 2)z_2\\
y_1+y_2+z_1z_2\\
z_1+z_2
\end{pmatrix}
$$
$$
(G,\cdot)\cong\left\{
 \begin{array}{cl}
 M_3(3),&~p=3\\
 M(p),&~p\neq 3\\
 \end{array}
 \right.
$$

\item Socle of order $p^3$:
$$
\begin{pmatrix}x_1\\y_1\\z_1\end{pmatrix}\cdot\begin{pmatrix}x_2\\y_2\\z_2\end{pmatrix}
:=\begin{pmatrix}
x_1+x_2\\
y_1+y_2\\
z_1+z_2
\end{pmatrix}
$$
$$
(G,\cdot)\cong (\Z/(p))^3
$$
\end{itemize}
\end{enumerate}
\end{theorem}

The next sections are devoted to the proof of these theorems, one section for every 
possible additive group.

\section{$(G,+)$ isomorphic to $\Z/(p^3)$}

When $p\neq 2$, we have the possible braces given by the multiplications $x\cdot y:=x+y$,  
$x\cdot y:=x+(1+px)y=x+y+pxy$, and $x\cdot y:=x+(1+p^2x)y=x+y+p^2xy$. In the three cases, we have 
$(G,\cdot)\cong \Z/(p^3)$ because $1$ is an element of order $p^3$ of the multiplicative group
of these braces. See \cite[Theorem~1 and Section~5]{Rump} for the missing details.

When $p=2$, we have the multiplications $x\cdot y:=x+y$,  
$x\cdot y:=x+(1+2x)y=x+y+2xy$, and $x\cdot y:=x+(1+2\alpha)^xy$ for 
$\alpha=1,2,3$. 
We have $(G,\cdot)\cong \Z/(8)$ in the first case, 
$(G,\cdot)\cong \Z/(2)\times\Z/(4)$ in the second case,
and, in the third case,
$(G,\cdot)\cong Q_8$ when $\alpha=1$,
$(G,\cdot)\cong\Z/(8)$ when $\alpha=2$, and
$(G,\cdot)\cong D_4$ when $\alpha=3$.
See \cite[Theorem~1 and Sections 6 and 7]{Rump} for the missing details.

\section{$(G,+)$ isomorphic to $\left(\Z/(p)\right)^3$}

When the socle is different from zero, $G/\soc(G)$ is a brace of 
order $\leq p^2$ with additive group isomorphic to $\Z/(p)$ or $(\Z/(p))^2$. 
By Proposition~\ref{bracesp2}, we know all the possible structures of 
brace over $G/\soc(G)$. For each of these possibilities, we apply Theorem~\ref{BenDavid} 
to find all the brace structures over $G$. To this end, we need to find all the faithful representations
$$\sigma:(G/\soc(G),\cdot)\hookrightarrow \aut(\left(\Z/(p)\right)^3)\cong GL_3(\Z/(p)),$$
and all the surjective morphisms $h:(G,+)\to (G/\soc(G),+)$ such that 
$h\circ \sigma(g)=\lambda_g\circ h$, where $\lambda_g(g')=g\cdot g'-g$
is the lambda map in $G/\soc(G)$. 
Then, some of these structures might give rise to isomorphic braces, so we have to 
determine the repeated cases.
Two of these structures are isomorphic if there exists an $F\in GL_3(\Z/(p))$
such that $\sigma'(h'(x,y,z))=F^{-1} \sigma(h(F(x,y,z))) F,$
for all $(x,y,z)\in(\Z/(p))^3$.

We begin by making a general simplification of the problem. 
Since the image of $G/\soc(G)$ by $\sigma$ is a $p$-group, 
it is contained in a Sylow $p$-subgroup of $GL_3(\Z/(p)).$ 
Observe that
$$\ord{GL_3(\Z/(p))}=(p^3-1)(p^3-p)(p^3-p^2)=p^3(p^3-1)(p^2-1)(p-1),$$
and thus a Sylow $p$-subgroup of $GL_3(\Z/(p))$ is
$$
T_p:=\left\lbrace 
\begin{pmatrix}
1 & a & b \\
0 & 1 & c\\
0 & 0 & 1
\end{pmatrix}\in GL_3(\Z/(p)):
a,b,c\in\Z/(p)\right\rbrace.
$$
We know that any two Sylow $p$-subgroups are conjugate,
and conjugate representations give rise to isomorphic braces, 
so we may take $\sigma: G/\soc(G)\hookrightarrow T_p$.

\subsection{Socle of order $p^3$}
Since $G=\soc(G)$, the only possible structure is the trivial one.

\subsection{Socle of order $p^2$}
$G/\soc(G)$ is a brace of order $p$, thus it is the trivial brace over $\Z/(p)$.
In this case, the morphism $\sigma$ is determined by a matrix $A\in T_p$ of order $p$, and the  
morphism $h$ is determined by a non-zero vector $(\alpha,\beta,\gamma)\in(\Z/(p))^3$ satisfying
$(\alpha,\beta,\gamma) A=(\alpha,\beta,\gamma)$. Then, the multiplication is given by
$$
\begin{pmatrix}
x_1\\
y_1\\
z_1
\end{pmatrix}\cdot
\begin{pmatrix}
x_2\\
y_2\\
z_2
\end{pmatrix}:=
\begin{pmatrix}
x_1\\
y_1\\
z_1
\end{pmatrix}+A^{\alpha x_1+\beta y_1 +\gamma z_1}
\begin{pmatrix}
x_2\\
y_2\\
z_2
\end{pmatrix}.
$$
Two of these structures, determined by $A$, $(\alpha,\beta,\gamma)$, and $A'$, $(\alpha',\beta',\gamma')$ 
respectively, are isomorphic if 
$$
A^{(\alpha,\beta,\gamma)(x,y,z)^t}=(F^{-1}A'F)^{(\alpha',~\beta',~\gamma')F(x,y,z)^t},
$$
for some matrix $F\in GL_3(\Z/(p)).$

Since $A$ has order $p$, $0=A^p-\Id=(A-\Id)^p$, and thus its minimal polynomial divides $(x-1)^p$. 
This implies that $A$ is conjugate to a matrix of Jordan form of eigenvalue 1, 
so we may take $A$ to be one of the following matrices:
$$
A=\begin{pmatrix}
1 & 1 & 0 \\
0 & 1 & 1\\
0 & 0 & 1
\end{pmatrix},\text{ or }
A=\begin{pmatrix}
1 & 1 & 0 \\
0 & 1 & 0\\
0 & 0 & 1
\end{pmatrix}.
$$
In the first case, $(\alpha,\beta,\gamma)=(0,0,k)$, $k\neq 0$, but using $F=k^{-1} \Id$, 
$A$ remains unchanged and the vector becomes  
$(\alpha,\beta,\gamma)=(0,0,1)$. In the second case, $(\alpha,\beta,\gamma)=(0,n,m)$. If $m\neq 0$, use 
$F=\begin{pmatrix}
1 & 0 & 0 \\
0 & 1 & 0\\
0 & -n/m & 1/m
\end{pmatrix}$
to obtain $(\alpha,\beta,\gamma)=(0,0,1)$ without changing $A$. If $m=0$, use $F=n^{-1} \Id$ to obtain
$(\alpha,\beta,\gamma)=(0,1,0)$ without changing $A$. In conclusion, we obtain three non-isomorphic braces:
\begin{enumerate}
\item $p\neq 2$, 
$\begin{pmatrix}
 x_1\\
 y_1\\
 z_1\\
\end{pmatrix}\cdot
\begin{pmatrix}
 x_2\\
 y_2\\
 z_2\\
\end{pmatrix}
:=\begin{pmatrix}
   x_1\\
   y_1\\
   z_1\\
  \end{pmatrix}
+\begin{pmatrix}
1 & 1 & 0 \\
0 & 1 & 1\\
0 & 0 & 1
\end{pmatrix}^{z_1}\begin{pmatrix}
x_2 \\
y_2\\
z_2
\end{pmatrix},$

\item $\begin{pmatrix}
 x_1\\
 y_1\\
 z_1\\
\end{pmatrix}\cdot
\begin{pmatrix}
 x_2\\
 y_2\\
 z_2\\
\end{pmatrix}
:=\begin{pmatrix}
   x_1\\
   y_1\\
   z_1\\
  \end{pmatrix}
+\begin{pmatrix}
1 & 1 & 0 \\
0 & 1 & 0\\
0 & 0 & 1
\end{pmatrix}^{y_1}\begin{pmatrix}
x_2 \\
y_2\\
z_2
\end{pmatrix},$

\item $\begin{pmatrix}
 x_1\\
 y_1\\
 z_1\\
\end{pmatrix}\cdot
\begin{pmatrix}
 x_2\\
 y_2\\
 z_2\\
\end{pmatrix}
:=\begin{pmatrix}
   x_1\\
   y_1\\
   z_1\\
  \end{pmatrix}+\begin{pmatrix}
1 & 1 & 0 \\
0 & 1 & 0\\
0 & 0 & 1
\end{pmatrix}^{z_1}\begin{pmatrix}
x_2 \\
y_2\\
z_2
\end{pmatrix}.$
\end{enumerate}

These three braces are non-isomorphic because the first case is a left brace
which is not a right brace, 
the second is a brace with abelian multiplicative group, and the third is a 
two-sided brace with non-abelian multiplicative group. 
The first case is not considered when $p=2$ because $A$ has order 4.

When $p\neq 2,3$, all the elements have order $p$, so we have $(G,\cdot)$ 
isomorphic to $M(p)$ in the first and the third cases, and to $(\Z/(p))^3$ in the second case.
When $p=3$, again we have $(G,\cdot)$ isomorphic to $(\Z/(p))^3$ in the second case,
and to $M(3)$ in the third case because all the elements have order $3$.
But, in the first case, $(G,\cdot)\cong M_3(3)$ because $(0,0,1)$ is an element of 
order $9$.
When $p=2$, we have $(G,\cdot)$ isomorphic to $\Z/(2)\times\Z/(4)$ in the second case ($(0,1,0)$ has order $4$,
and there are no elements of order $8$), and isomorphic to $D_4$ in the third case ($(1,0,0)$ and $(0,1,0)$
are two elements of order $2$).

\subsection{Socle of order $p$}
$G/\soc(G)$ is a brace of order $p^2$ with $(G/\soc(G),+)\cong (\Z/(p))^2$, and the possible structures over $G$ depend
on the structure over $G/\soc(G)$, which we have classified in Proposition~\ref{bracesp2}. In fact, $G/\soc(G)$ can only
have a structure of type (iv) or of type (v) of Proposition~\ref{bracesp2}.

\subsubsection{$G/\soc(G)$ is of type (iv)}
In this case, the morphism $\sigma$ is determined by
two matrices $A$ and $B$ in $GL_3(\Z/(p))$ of order $p$ which commute, 
and $h$ is determined by two linearly independent 
vectors $h_1$ and $h_2$ of $(\Z/(p))^3$ such that 
$h_i A=h_i$ and $h_i B=h_i$ 
(i.e. they are common eigenvectors of eigenvalue 1 of $A^t$ and of $B^t$). 
Then, the multiplication is given by
$$
\begin{pmatrix}
 x_1\\
 y_1\\
 z_1\\
\end{pmatrix}\cdot
\begin{pmatrix}
 x_2\\
 y_2\\
 z_2\\
\end{pmatrix}
:=\begin{pmatrix}
   x_1\\
   y_1\\
   z_1\\
  \end{pmatrix}+A^{h_1 (x_1,y_1,z_1)^t}B^{h_2 (x_1,y_1,z_1)^t}
\begin{pmatrix}
x_2\\
y_2\\
z_2
\end{pmatrix}.
$$
Two of these structures are isomorphic if 
$$
A^{h_1 (x,y,z)^t}B^{h_2 (x,y,z)^t}=(F^{-1}A'F)^{h'_1 F (x,y,z)^t}(F^{-1}B'F)^{h'_2 F (x,y,z)^t},
$$
for some $F\in GL_3(\Z/(p)).$

Since 
$A^t$ and $B^t$ have two common linearly independent eigenvectors, 
using a linear change of basis,
we can take $h_1=(0,1,0)$, $h_2=(0,0,1)$, and
$A=\begin{pmatrix}
1 & a & b\\
0 & 1 & 0\\
0 & 0 & 1
\end{pmatrix}$, 
$B=\begin{pmatrix}
1 & c & d\\
0 & 1 & 0\\
0 & 0 & 1
\end{pmatrix}.$
Observe that the $(1,1)$-entry of $A$ and $B$ is 1 because they both must have order $p$.
So we can make the following rearrangement of the multiplication
\begin{eqnarray*}
\begin{pmatrix}
 x_1\\
 y_1\\
 z_1\\
\end{pmatrix}\cdot
\begin{pmatrix}
 x_2\\
 y_2\\
 z_2\\
\end{pmatrix}
&=&\begin{pmatrix}
   x_1\\
   y_1\\
   z_1\\
  \end{pmatrix}+A^{y_1}B^{z_1}
\begin{pmatrix}
x_2\\
y_2\\
z_2
\end{pmatrix}
=\begin{pmatrix}
  x_1\\
  y_1\\
  z_1\\
 \end{pmatrix}
+
\begin{pmatrix}
1 & ay_1+cz_1 & by_1+dz_1\\
0 & 1 & 0\\
0 & 0 & 1
\end{pmatrix}
\begin{pmatrix}
x_2\\
y_2\\
z_2
\end{pmatrix}\\[6pt]
&=&\begin{pmatrix}
  x_1\\
  y_1\\
  z_1\\
 \end{pmatrix}
+
\begin{pmatrix}
1 & 1 & 0\\
0 & 1 & 0\\
0 & 0 & 1
\end{pmatrix}^{ay_1+cz_1}
\begin{pmatrix}
1 & 0 & 1\\
0 & 1 & 0\\
0 & 0 & 1
\end{pmatrix}^{by_1+dz_1}
\begin{pmatrix}
x_2\\
y_2\\
z_2
\end{pmatrix}.
\end{eqnarray*}
Observe that the vectors $(a,c)$ and $(b,d)$ must be linearly independent for the brace to have 
socle of order $p$. Then, after this rearrangement, 
the isomorphism condition becomes
$$
F\begin{pmatrix}
1 & 1 & 0\\
0 & 1 & 0\\
0 & 0 & 1
\end{pmatrix}^{a'y+c'z}
\begin{pmatrix}
1 & 0 & 1\\
0 & 1 & 0\\
0 & 0 & 1
\end{pmatrix}^{b'y+d'z}=
\begin{pmatrix}
1 & 1 & 0\\
0 & 1 & 0\\
0 & 0 & 1
\end{pmatrix}^{(0,a,c)F(x,y,z)^t}
\begin{pmatrix}
1 & 0 & 1\\
0 & 1 & 0\\
0 & 0 & 1
\end{pmatrix}^{(0,b,d)F(x,y,z)^t}F.
$$
Assume $F=(f_{ij})_{i,j}$. Computing explicitly the right side, we get
\begin{eqnarray*}
F\begin{pmatrix}
1 & 1 & 0\\
0 & 1 & 0\\
0 & 0 & 1
\end{pmatrix}^{a'y+c'z}
\begin{pmatrix}
1 & 0 & 1\\
0 & 1 & 0\\
0 & 0 & 1
\end{pmatrix}^{b'y+d'z}
=\begin{pmatrix}
f_{11} & f_{11}(0,a',c')\begin{pmatrix}x\\y\\z\end{pmatrix}+f_{12} \hspace{6pt}& f_{11}(0,b',d')\begin{pmatrix}x\\y\\z\end{pmatrix}+f_{13}\vspace{6pt}\\
f_{21} & f_{22}+f_{21}(0,a',c')\begin{pmatrix}x\\y\\z\end{pmatrix} & f_{23}+f_{21}(0,b',d')\begin{pmatrix}x\\y\\z\end{pmatrix}\vspace{6pt}\\
f_{31} & f_{32}+f_{31}(0,a',c')\begin{pmatrix}x\\y\\z\end{pmatrix} & f_{33}+f_{31}(0,b',d')\begin{pmatrix}x\\y\\z\end{pmatrix}
\end{pmatrix}.
\end{eqnarray*}
Doing the same with the left side, we obtain
\begin{eqnarray*}
\lefteqn{\begin{pmatrix}
1 & 1 & 0\\
0 & 1 & 0\\
0 & 0 & 1
\end{pmatrix}^{\begin{matrix}(0,a,c) F\end{matrix}\begin{pmatrix}x\\y\\z\end{pmatrix}}
\begin{pmatrix}
1 & 0 & 1\\
0 & 1 & 0\\
0 & 0 & 1
\end{pmatrix}^{\begin{matrix}(0,b,d) F\end{matrix}\begin{pmatrix}x\\y\\z\end{pmatrix}}F}\\[7pt]
&=&\begin{pmatrix}
f_{11} \hspace{9pt}& f_{12}+f_{22}(0,a,c) F\begin{pmatrix}x\\y\\z\end{pmatrix}+f_{32}(0,b,d) F\begin{pmatrix}x\\y\\z\end{pmatrix} 
\hspace{9pt}& f_{13}+f_{23}(0,a,c) F\begin{pmatrix}x\\y\\z\end{pmatrix}+f_{33}(0,b,d) F\begin{pmatrix}x\\y\\z\end{pmatrix}\vspace{6pt}\\
f_{21} & f_{22} & f_{23}\\
f_{31} & f_{32} & f_{33}
\end{pmatrix}.
\end{eqnarray*}
Equating each entry of the two matrices, the second and the third row tell us that $f_{21}=f_{31}=0$. The first row tells us that
$$
f_{11}\begin{pmatrix}
a' & c' \\
b' & d' \\
\end{pmatrix}=
G^t\begin{pmatrix}
a & c \\
b & d \\
\end{pmatrix} G,
$$
where $G=\begin{pmatrix}
f_{22} & f_{23} \\
f_{32} & f_{33} \\
\end{pmatrix}\in GL_2(\Z/(p))$. In particular, $f_{12}$ and $f_{13}$ do not appear anywhere, so we may take $f_{12}=f_{13}=0$.

Using $f_{11}$, we can multiply the matrix $\begin{pmatrix}a&b\\c&d\\\end{pmatrix}$ by any element 
of $\Z/(p)$ different from zero. On the other hand, to see the effect of the matrix $G$, we will consider
the different elementary matrices, which are generators of $GL_2(\Z/(p))$.
\begin{enumerate}[(i)]
\item When $
G=\begin{pmatrix}
0 & 1 \\
1 & 0 \\
\end{pmatrix},~~G^t\begin{pmatrix}
a & c \\
b & d \\
\end{pmatrix} G=
\begin{pmatrix}
a & b \\
c & d \\
\end{pmatrix}.
$
\item When $
G=\begin{pmatrix}
\mu & 0 \\
0 & 1 \\
\end{pmatrix},~\mu\in\Z/(p)\setminus\{0\},~~G^t\begin{pmatrix}
a & c \\
b & d \\
\end{pmatrix} G=
\begin{pmatrix}
\mu^2 a & \mu c \\
\mu b & d \\
\end{pmatrix}.
$
\item When $
G=\begin{pmatrix}
1 & 0 \\
0 & \mu \\
\end{pmatrix},~\mu\in\Z/(p)\setminus\{0\},~~G^t\begin{pmatrix}
a & c \\
b & d \\
\end{pmatrix} G=
\begin{pmatrix}
a & \mu c \\
\mu b & \mu^2 d \\
\end{pmatrix}.
$
\item When $
G=\begin{pmatrix}
1 & \mu \\
0 & 1 \\
\end{pmatrix},~\mu\in\Z/(p),~~G^t\begin{pmatrix}
a & c \\
b & d \\
\end{pmatrix} G=
\begin{pmatrix}
a & c+\mu a \\
b+\mu a & d+\mu(b+c)+\mu^2 a \\
\end{pmatrix}.
$
\item When $
G=\begin{pmatrix}
1 & 0 \\
\mu & 1 \\
\end{pmatrix},~\mu\in\Z/(p),~~G^t\begin{pmatrix}
a & c \\
b & d \\
\end{pmatrix} G=
\begin{pmatrix}
a+\mu(b+c)+\mu^2 d & c+\mu d \\
b+\mu d & d \\
\end{pmatrix}.
$
\end{enumerate}

The extreme case is $a=d=0$ and $c=-b$. Then we can use $f_{11}$ to obtain $b=-1$ and $c=1$. In all other 
cases, we can use changes of type (i), (iv) and (v) to turn $b$ into 0. Then, we use $f_{11}$ to turn 
$d$ into 1, and changes of type (ii) to turn $a$ into 1 or $\varepsilon$, depending if $a$ is either a 
square or not. The value of $c$ can be changed to $-c$ using a change of type (iii) with $\mu=-1$.

When $p=2$, there are the following non-isomorphic cases
\begin{enumerate}
\item $
\begin{pmatrix}
a & c \\
b & d \\
\end{pmatrix}=
\begin{pmatrix}
0 & 1 \\
1 & 0 \\
\end{pmatrix},
$
\item $
\begin{pmatrix}
a & c \\
b & d \\
\end{pmatrix}=
\begin{pmatrix}
1 & 0 \\
0 & 1 \\
\end{pmatrix},
$
\item $
\begin{pmatrix}
a & c \\
b & d \\
\end{pmatrix}=
\begin{pmatrix}
1 & 1 \\
0 & 1 \\
\end{pmatrix}.
$
\end{enumerate}

When $p\neq 2$, fix an element $\varepsilon$ which is not a square in $\Z/(p)$. 
Then, there are the following non-isomorphic cases
\begin{enumerate}
\item $
\begin{pmatrix}
a & c \\
b & d \\
\end{pmatrix}=
\begin{pmatrix}
0 & -1 \\
1 & 0 \\
\end{pmatrix},
$
\item $
\begin{pmatrix}
a & c \\
b & d \\
\end{pmatrix}=
\begin{pmatrix}
1 & \lambda \\
0 & 1 \\
\end{pmatrix},
$
for each
$
\lambda\in\{0,1,\dots,\frac{p-1}{2}\},
$
\item $
\begin{pmatrix}
a & c \\
b & d \\
\end{pmatrix}=
\begin{pmatrix}
\varepsilon & \lambda \\
0 & 1 \\
\end{pmatrix},
$
for each
$
\lambda\in\{0,1,\dots,\frac{p-1}{2}\}.
$
\end{enumerate}
We assert that these three families give rise to non-isomorphic braces. Changes on the first family keeps the 
first entry always equal to zero, so it is non-isomorphic with the other two families. Compare now the determinant 
of a change of the second family and of the third family:
$$
\det\left(f_{11}^{-1}G^t\begin{pmatrix}1&\lambda\\0&1\end{pmatrix}G\right)=
f_{11}^{-2}(\det G)^2\text{, which is a square,}
$$
$$
\det\left(f_{11}^{-1}G^t\begin{pmatrix}\varepsilon &\lambda\\0&1\end{pmatrix}G\right)=
\varepsilon f_{11}^{-2}(\det G)^2\text{, which is not a square,}
$$
so they must give non-isomorphic braces. Finally, we see that matrices of the second family with 
different values of $\lambda$ give non-isomorphic braces.
Take two elements $\lambda$ and $\lambda'$ of $\{0,1,\dots,(p-1)/2\}$. If they represent
two isomorphic braces, then 
$$
f_{11}\begin{pmatrix}1& \lambda'\\0& 1\end{pmatrix}=G^t\begin{pmatrix}1& \lambda\\0& 1\end{pmatrix}G,
$$
and comparing determinants, we see that $f_{11}=\pm\det G$. But, multiplying the matrices explicitly, we get
$$
\begin{pmatrix}
f_{11}& \lambda' f_{11}\\
0& f_{11}
\end{pmatrix}=
\begin{pmatrix}
\star& f_{22}f_{23}+\lambda f_{22}f_{33}+f_{32}f_{33}\\
f_{22}f_{23}+\lambda f_{23}f_{32}+f_{32}f_{33}& \star\\
\end{pmatrix},
$$
and subtracting the two shown entries, we get $\lambda' f_{11}=\lambda(f_{22}f_{33}-f_{23}f_{32})=\lambda\det G$.
Thus $\lambda=\pm\lambda'$, and if $\lambda,\lambda'\in\{0,1,\dots,(p-1)/2\}$, then $\lambda=\lambda'$. 
By an analogous argument, matrices of the third family with different $\lambda$'s define 
non-isomorphic braces structures.

Now we compute the multiplicative group. When $p=2$,
\begin{enumerate}[1.]
\item Since $\begin{pmatrix}x\\y\\z\end{pmatrix}^n=\begin{pmatrix}nx\\ny\\nz\end{pmatrix},$ the exponent is $2$, so $(G,\cdot)\cong (\Z/(2))^3$.
\item The multiplication is commutative, and the exponent is $4$, so $(G,\cdot)\cong\Z/(2)\times\Z/(4)$.
\item The multiplication is noncommutative, and $(1,0,0)$ is the unique element of order $2$, so $(G,\cdot)\cong Q$.
\end{enumerate}
When $p\neq 2$, the exponent is always $p$. Then,
\begin{enumerate}[1.]
\item The multiplication is noncommutative, so $(G,\cdot)\cong M(p)$.
\item The multiplication is commutative if and only if $\lambda=0$, so $(G,\cdot)$ is isomorphic to 
$(\Z/(p))^3$ when $\lambda=0$, and to $M(p)$ when $\lambda\neq 0$.
\item The multiplication is commutative if and only if $\lambda=0$, so $(G,\cdot)$ is isomorphic to 
$(\Z/(p))^3$ when $\lambda=0$, and to $M(p)$ when $\lambda\neq 0$.
\end{enumerate}

\subsubsection{$G/\soc(G)$ is of type (v)}
\paragraph{Case $p\neq 2$.}

In this case, $\sigma$ is determined by two matrices $A$ and $B$ in $GL_3(\Z/(p))$ of order $p$ 
such that $AB=BA$ in the following way
$$
\begin{array}{cccc}
\sigma:& (G/\soc(G),\cdot)\cong (\Z/(p))^2&\longrightarrow &GL_3(\Z/(p))\\
		&(1,0) &\mapsto &A\\
		&(0,1) &\mapsto &B\\
		&(x,y)=(1,0)^{x-C(y,2)}(0,1)^y &\mapsto &A^{x-C(y,2)}B^y.
\end{array}
$$
On the other hand, $h$ is determined by two linearly independent vectors $h_1$ and 
$h_2$ of $(\Z/(p))^3$ such that $h_1 A=h_1$, $h_2 A=h_2$, $h_1 B=h_1+h_2$, $h_2 B=h_2$.

By the properties of $h_1$ and $h_2$, and taking into account that 
$A$ and $B$ have order $p$, after a change of basis, we reduce to the case
$h_1=(0,1,0)$, $h_2=(0,0,1)$,
$A=\begin{pmatrix}
1 & a & b\\
0 & 1 & 0\\
0 & 0 & 1
\end{pmatrix}$ and
$B=\begin{pmatrix}
1 & c & d\\
0 & 1 & 1\\
0 & 0 & 1
\end{pmatrix}.$ But these two matrices commute if and only if $a=0$. Now, $b\neq 0$, so using the matrix
$F=\begin{pmatrix}
b & d & 0\\
0 & 1 & 0\\
0 & 0 & 1
\end{pmatrix},$ we can reduce to 
$h_1=(0,1,0)$, $h_2=(0,0,1)$,
$A=\begin{pmatrix}
1 & 0 & 1\\
0 & 1 & 0\\
0 & 0 & 1
\end{pmatrix}$ and
$B=\begin{pmatrix}
1 & c & 0\\
0 & 1 & 1\\
0 & 0 & 1
\end{pmatrix}.$
Then, the multiplication is given by
\begin{eqnarray*}
\begin{pmatrix}
 x_1\\
 y_1\\
 z_1\\
\end{pmatrix}\cdot
\begin{pmatrix}
 x_2\\
 y_2\\
 z_2\\
\end{pmatrix}
&:=&\begin{pmatrix}
   x_1\\
   y_1\\
   z_1\\
  \end{pmatrix}+\sigma(h(x_1,y_1,z_1))\begin{pmatrix}x_2\\y_2\\z_2\end{pmatrix}=
  \begin{pmatrix}
   x_1\\
   y_1\\
   z_1\\
  \end{pmatrix}+\sigma(y_1,z_1)\begin{pmatrix}x_2\\y_2\\z_2\end{pmatrix}\\[7pt]
&=&
\begin{pmatrix}
   x_1\\
   y_1\\
   z_1\\
  \end{pmatrix}+A^{y_1-C(z_1, 2)}B^{z_1}
\begin{pmatrix}
x_2\\
y_2\\
z_2
\end{pmatrix}.
\end{eqnarray*}

Two of this structures are isomorphic if 
$$
FA^{y-C(z,2)}B^{z}=
A^{(0,1,0) F (x,y,z)^t-C((0,0,1) F (x,y,z)^t,~ 2)}{B'}^{(0,0,1) F (x,y,z)^t}F.
$$
Assume $F=(f_{ij})$, $B=\begin{pmatrix}
1 & c & 0\\
0 & 1 & 1\\
0 & 0 & 1
\end{pmatrix}$, 
$B'=\begin{pmatrix}
1 & b & 0\\
0 & 1 & 1\\
0 & 0 & 1
\end{pmatrix}.$ 
Putting $(x,y,z)=(1,0,0)$, the isomorphism condition becomes
$$
\Id=A^{f_{21}-C(f_{31}, 2)}{B'}^{f_{31}},
$$
and we obtain $f_{21}=f_{31}=0$. 
For $(x,y,z)=(0,1,0)$,
$$
F A=A^{f_{22}-C(f_{32}, 2)}{B'}^{f_{32}}F,
$$ 

$$
F A=
\begin{pmatrix}
f_{11} & f_{12} & f_{11}+f_{13}\\
0 & f_{22} & f_{23}\\
0 & f_{32} & f_{33}
\end{pmatrix},
$$

\begin{eqnarray*}
A^{f_{22}-C(f_{32}, 2)}{B'}^{f_{32}}F&=&
\begin{pmatrix}
1 & bf_{32} & f_{22}-C(f_{32}, 2)+b C(f_{32}, 2)\\
0 & 1 & f_{32}\\
0 & 0 & 1
\end{pmatrix}
\begin{pmatrix}
f_{11} & f_{12} & f_{13}\\
0 & f_{22} & f_{23}\\
0 & f_{32} & f_{33}
\end{pmatrix}\\[6pt]
&=&\begin{pmatrix}
f_{11} & \star & f_{13}+f_{22}f_{33}\\
0 & f_{22}+f_{32}^2 & f_{23}+f_{32}f_{33}\\
0 & f_{32} & f_{33}
\end{pmatrix},
\end{eqnarray*}
and we obtain $f_{32}=0$ and $f_{11}=f_{22}f_{33}$. 
For $(x,y,z)=(0,0,1)$, 
$$
F B=A^{f_{23}-C(f_{33}, 2)}{B'}^{f_{33}}F,
$$ 

$$
F B=
\begin{pmatrix}
f_{11} & f_{12}+cf_{11} & \star\\
0 & f_{22} & \star\\
0 & 0 & \star
\end{pmatrix},
$$

$$
A^{f_{23}-C(f_{33}, 2)}{B'}^{f_{33}}F=
\begin{pmatrix}
1 & bf_{33} & f_{23}-C(f_{33}, 2)+b C(f_{33}, 2)\\
0 & 1 & f_{33}\\
0 & 0 & 1
\end{pmatrix}
\begin{pmatrix}
f_{11} & f_{12} & f_{13}\\
0 & f_{22} & f_{23}\\
0 & 0 & f_{33}
\end{pmatrix}
=\begin{pmatrix}
f_{11} & f_{12}+bf_{22}f_{33} & \star\\
0 & f_{22} & \star\\
0 & 0 & \star
\end{pmatrix},
$$
and we obtain $cf_{11}=bf_{22}f_{33}$.
Then we have that $F$ is a upper-triangular matrix, and we need 
$f_{11},f_{22},f_{33}\neq 0$ for $F$ to be invertible. 
Dividing $cf_{11}=bf_{22}f_{33}$ by $f_{11}=f_{22}f_{33}$, we obtain $b=c$.

In conclusion, for each $c\in\Z/(p)$, there is a left brace with multiplication
$$
\begin{pmatrix}
 x_1\\
 y_1\\
 z_1\\
\end{pmatrix}\cdot
\begin{pmatrix}
 x_2\\
 y_2\\
 z_2\\
\end{pmatrix}
:=\begin{pmatrix}
   x_1\\
   y_1\\
   z_1\\
  \end{pmatrix}+A^{y_1-C(z_1, 2)}B^{z_1}
\begin{pmatrix}
x_2\\
y_2\\
z_2
\end{pmatrix}
=\begin{pmatrix}
  x_1\\
  y_1\\
  z_1\\
 \end{pmatrix}
+
\begin{pmatrix}
1 & cz_1 & y_1+(c-1)C(z_1, 2)\\
0 & 1 & z_1\\
0 & 0 & 1
\end{pmatrix}
\begin{pmatrix}
x_2\\
y_2\\
z_2
\end{pmatrix}.
$$
All these braces are mutually non-isomorphic, and all the braces with these properties are of this form.

This multiplication is commutative if and only if $c=1$. When $p=3$, if $c=0$, the exponent is $3$, so $(G,\cdot)\cong M(3)$.
If $c\neq 0$, there are elements of order $9$ (for instance, $(0,0,1)$), and then $(G,\cdot)$ is isomorphic 
to $\Z/(3)\times\Z/(9)$ when $c=1$, and 
to $M_3(3)$ when $c\neq 0,1$. When $p\neq 3$, $(G,\cdot)$ has exponent $p$, and then it is isomorphic to $(\Z/(p))^3$
when $c=1$, and to $M(p)$ when $c\neq 1$.

\paragraph{Case $p=2$.} In this case, $(G/\soc(G),\cdot)\cong \Z/(4)$ and $(G/\soc(G),+)\cong \Z/(2)\times\Z/(2)$.
First of all, we need a matrix $A\in GL_3(\Z/(2))$ of order 4. Using the isomorphism condition, $A$ may be taken of Jordan form. 
The only matrix of order $4$ of this type is 
$A=\begin{pmatrix}
1 & 1 & 0\\
0 & 1 & 1\\
0 & 0 & 1
\end{pmatrix}.$

On the other hand, we need two vectors $h_1$ and $h_2$ of $(\Z/(2))^3$ such that $h_1 A=h_1+h_2$ and 
$h_2 A=h_2$. This is only possible for our $A$
if $h_1=(0,1,k)$ and $h_2=(0,0,1)$. But, when $k=1$, using
$F=\begin{pmatrix}
1 & 1 & 0\\
0 & 1 & 1\\
0 & 0 & 1
\end{pmatrix}
$, $A$ remains unchanged and $h_1$ and $h_2$ become $h_1=(0,1,0)$ and $h_2=(0,0,1)$.

In conclusion, there is a unique left brace with these properties up to isomorphism, with multiplication
$$
\begin{pmatrix}
 x_1\\
 y_1\\
 z_1\\
\end{pmatrix}\cdot 
\begin{pmatrix}
 x_2\\
 y_2\\
 z_2\\
\end{pmatrix}
:=\begin{pmatrix}
   x_1\\
   y_1\\
   z_1\\
  \end{pmatrix}+
\begin{pmatrix}
1 & z_1 & y_1\\
0 & 1 & z_1\\
0 & 0 & 1
\end{pmatrix}
\begin{pmatrix}
x_2 \\
y_2 \\
z_2
\end{pmatrix}.
$$

This multiplication is commutative, and has elements of order $4$, like $(0,0,1)$,
 but not of order $8$, so $(G,\cdot)\cong\Z/(2)\times\Z/(4)$.

\subsection{Trivial socle}
When the socle is trivial, the morphism $\lambda: G\to GL_3(\Z/(p))$
is injective. Since $\lambda(G)$ is a $p$-group, it is contained in a Sylow $p$-subgroup of 
$GL_3(\Z/(p))$, and we can take this subgroup to be $T_p$. But $\ord{T_p}=p^3$, so in fact $\lambda$
is an isomorphism to $T_p$. We are done if we can 
find a bijective map $\pi:T_p\to (\Z/(p))^3$ such that 
$\pi(AB)=\pi(A)+A\pi(B)$ for all $A$ and $B$ in $T_p$.

Suppose that $\pi^{-1}(1,0,0)=
\begin{pmatrix}
1 & a & b\\
0 & 1 & c\\
0 & 0 & 1
\end{pmatrix}.
$
It cannot happen that $\begin{pmatrix}1\\0\\0\end{pmatrix}\cdot \begin{pmatrix}x\\y\\z\end{pmatrix}=\begin{pmatrix}1\\0\\0\end{pmatrix}+\lambda_{(1,0,0)}(x,y,z)=\begin{pmatrix}x\\y\\z\end{pmatrix}$ for some $\begin{pmatrix}x\\y\\z\end{pmatrix}$. 
Equivalently, the system of linear equations on $x$, $y$ and $z$
$$
\begin{pmatrix}
1\\
0\\
0
\end{pmatrix}+\begin{pmatrix}
0 & a & b\\
0 & 0 & c\\
0 & 0 & 0
\end{pmatrix}
\begin{pmatrix}
x\\
y\\
z
\end{pmatrix}=0
$$
cannot have solutions. This gives $a=0$ and $c\neq 0$. 

To compute the matrix corresponding to the vector $\begin{pmatrix}x\\0\\0\end{pmatrix}$, observe that 
$\begin{pmatrix}1\\0\\0\end{pmatrix}+\begin{pmatrix}1\\0\\0\end{pmatrix}=
\begin{pmatrix}1\\0\\0\end{pmatrix}\cdot \begin{pmatrix}1\\0\\0\end{pmatrix}$ 
because
$$
\begin{pmatrix}1\\0\\0\end{pmatrix}+\begin{pmatrix}1\\0\\0\end{pmatrix}=\begin{pmatrix}1\\0\\0\end{pmatrix}\cdot\lambda^{-1}_{(1,0,0)}(1,0,0)=\begin{pmatrix}1\\0\\0\end{pmatrix}\cdot 
\left(
\begin{pmatrix}
1 & 0 & -b\\
0 & 1 & -c\\
0 & 0 & 1
\end{pmatrix}
\begin{pmatrix}
1\\
0\\
0
\end{pmatrix}
\right)
=\begin{pmatrix}1\\0\\0\end{pmatrix}\cdot \begin{pmatrix}1\\0\\0\end{pmatrix}.
$$
Then, 
$$
\begin{pmatrix}x\\0\\0\end{pmatrix}=\begin{pmatrix}1\\0\\0\end{pmatrix}+\overset{(x}{\cdots}+\begin{pmatrix}1\\0\\0\end{pmatrix}=
\begin{pmatrix}1\\0\\0\end{pmatrix}\cdot\overset{(x}{\dots}\cdot \begin{pmatrix}1\\0\\0\end{pmatrix}.
$$
Therefore,
$$
\pi^{-1}(x,0,0)=
\begin{pmatrix}
1 & 0 & b\\
0 & 1 & c\\
0 & 0 & 1
\end{pmatrix}^x=
\begin{pmatrix}
1 & 0 & bx\\
0 & 1 & cx\\
0 & 0 & 1
\end{pmatrix}.
$$

Assume now that 
$\pi^{-1}(0,1,0)=\begin{pmatrix}
1 & a' & b'\\
0 & 1 & c'\\
0 & 0 & 1
\end{pmatrix}.
$
The condition $\begin{pmatrix}0\\1\\0\end{pmatrix}\cdot\begin{pmatrix}x\\y\\z\end{pmatrix}\neq \begin{pmatrix}x\\y\\z\end{pmatrix}$,
 for all $\begin{pmatrix}x\\y\\z\end{pmatrix}$, gives $a'c'=0$.
On the other hand,
$$
\begin{pmatrix}1\\1\\0\end{pmatrix}=\begin{pmatrix}1\\0\\0\end{pmatrix}+\begin{pmatrix}0\\1\\0\end{pmatrix}=
\begin{pmatrix}1\\0\\0\end{pmatrix}\cdot \begin{pmatrix}0\\1\\0\end{pmatrix},
$$
so
$$
\pi^{-1}(1,1,0)=
\begin{pmatrix}
1 & 0 & b\\
0 & 1 & c\\
0 & 0 & 1
\end{pmatrix}
\begin{pmatrix}
1 & a' & b'\\
0 & 1 & c'\\
0 & 0 & 1
\end{pmatrix}=
\begin{pmatrix}
1 & a' & b+b'\\
0 & 1 & c+c'\\
0 & 0 & 1
\end{pmatrix}.
$$
The condition $(1,1,0)\cdot (x,y,z)\neq (x,y,z)$ for all $(x,y,z)$ gives $a'(c+c')=0$. 
But we know that $a'c'=0$ and $c\neq 0$, so we obtain $a'=0$ and $b'\neq 0$.
To compute $\pi^{-1}(0,y,0)$, observe that $(0,1,0)+(0,1,0)=(0,1,0)\cdot (0,1,0)$ 
because 
$$
\begin{pmatrix}0\\1\\0\end{pmatrix}+\begin{pmatrix}0\\1\\0\end{pmatrix}=\begin{pmatrix}0\\1\\0\end{pmatrix}\cdot\lambda^{-1}_{(0,1,0)}(0,1,0)=\begin{pmatrix}0\\1\\0\end{pmatrix}\cdot
\left(
\begin{pmatrix}
1 & 0 & -b'\\
0 & 1 & -c'\\
0 & 0 & 1
\end{pmatrix}
\begin{pmatrix}
0\\
1\\
0
\end{pmatrix}
\right)
=\begin{pmatrix}0\\1\\0\end{pmatrix}\cdot \begin{pmatrix}0\\1\\0\end{pmatrix}.
$$
Then,
$$
(0,y,0)=(0,1,0)+\overset{(y}{\cdots}+(0,1,0)=(0,1,0)\cdot\overset{(y}{\dots}\cdot (0,1,0).
$$
Therefore,
$$
\pi^{-1}(0,y,0)=
\begin{pmatrix}
1 & 0 & b'\\
0 & 1 & c'\\
0 & 0 & 1
\end{pmatrix}^y=
\begin{pmatrix}
1 & 0 & b'y\\
0 & 1 & c'y\\
0 & 0 & 1
\end{pmatrix}.
$$

To compute the matrix corresponding to $(x,y,0)$, observe that 
$$(x,y,0)=(x,0,0)+(0,y,0)=(x,0,0)\cdot \lambda^{-1}_{(x,0,0)}(0,y,0)=(x,0,0)\cdot (0,y,0).$$
Thus $(x,y,0)$ has to be assigned to the matrix
$$
\begin{pmatrix}
1 & 0 & xb\\
0 & 1 & xc\\
0 & 0 & 1
\end{pmatrix}
\begin{pmatrix}
1 & 0 & yb'\\
0 & 1 & yc'\\
0 & 0 & 1
\end{pmatrix}=
\begin{pmatrix}
1 & 0 & xb+yb'\\
0 & 1 & xc+yc'\\
0 & 0 & 1
\end{pmatrix}.
$$
Conversely,
$$
\pi\begin{pmatrix}
1 & 0 & x\\
0 & 1 & y\\
0 & 0 & 1
\end{pmatrix}=
\left( 
\left(\begin{pmatrix}
b & b' \\
c & c' \\
\end{pmatrix}^{-1}
\begin{pmatrix}
x \\
y \\
\end{pmatrix}\right)^t,0
\right)=
(K(c'x-b'y),K(by-cx),0),$$
where $K=(bc'-b'c)^{-1}$.

If
$
\pi\begin{pmatrix}
1 & 1 & 0\\
0 & 1 & 0\\
0 & 0 & 1
\end{pmatrix}
=(x_0,y_0,z_0)$ with $z_0\neq 0$, we can compute all the other values of $\pi$
as follows
\begin{eqnarray*}
\pi\begin{pmatrix}
1 & n & x\\
0 & 1 & y\\
0 & 0 & 1
\end{pmatrix}&=&
\pi\left(\begin{pmatrix}
1 & 0 & x\\
0 & 1 & y\\
0 & 0 & 1
\end{pmatrix}
\begin{pmatrix}
1 & n & 0\\
0 & 1 & 0\\
0 & 0 & 1
\end{pmatrix}\right)\\[6pt]
&=&\pi\left(\begin{pmatrix}
1 & 0 & x\\
0 & 1 & y\\
0 & 0 & 1
\end{pmatrix}\right)+
\begin{pmatrix}
1 & 0 & x\\
0 & 1 & y\\
0 & 0 & 1
\end{pmatrix}
\begin{pmatrix}
nx_0+C(n, 2)y_0 \\
ny_0 \\
nz_0 
\end{pmatrix}\\[6pt]
&=&\begin{pmatrix}
(c'x-b'y)K+nx_0+C(n, 2)y_0+xnz_0\\
(by-cx)K+ny_0+nyz_0\\
nz_0
\end{pmatrix},
\end{eqnarray*}
with $K=(bc'-b'c)^{-1}\neq 0$ and $b',c,z_0\neq 0$.
Summarizing, we have used some necessary conditions in order to find 
res\-trictions to the possible definitions of $\pi$.
We will see now if any of these maps give rise to a left brace. The answer depends on $p$.

When $p\neq 2$, $\pi$ does not satisfy $\pi(AB)=\pi(A)+A\pi(B)$. Take first
$
A_1=\begin{pmatrix}
1 & 1 & 0\\
0 & 1 & 0\\
0 & 0 & 1
\end{pmatrix}
$ and
$B_1=\begin{pmatrix}
1 & 0 & 0\\
0 & 1 & 1\\
0 & 0 & 1
\end{pmatrix}
$. Then, 
$$
\pi(A_1 B_1)=\pi
\begin{pmatrix}
1 & 1 & 1\\
0 & 1 & 1\\
0 & 0 & 1
\end{pmatrix}=
\begin{pmatrix}
(c'-b')K+x_0+z_0\\
(b-c)K+y_0+z_0\\
z_0
\end{pmatrix},
$$
and, on the other side,
$$
\pi(A_1)+A_1\pi(B_1)=
\begin{pmatrix}
x_0\\
y_0\\
z_0
\end{pmatrix}+
\begin{pmatrix}
1 & 1 & 0\\
0 & 1 & 0\\
0 & 0 & 1
\end{pmatrix}
\begin{pmatrix}
-b'K \\
bK \\
0 
\end{pmatrix}=
\begin{pmatrix}
x_0-b'K+bK\\
y_0+bK\\
z_0
\end{pmatrix}.
$$
Looking at the second component, we see that $z_0$ must be equal to $cK$.

Now, take
$
A_2=\begin{pmatrix}
1 & 1 & 0\\
0 & 1 & 0\\
0 & 0 & 1
\end{pmatrix}
$ and
$B_2=\begin{pmatrix}
1 & 0 & 1\\
0 & 1 & 0\\
0 & 0 & 1
\end{pmatrix}
$. Then, 
$$
\pi(A_2 B_2)=\pi
\begin{pmatrix}
1 & 1 & 1\\
0 & 1 & 0\\
0 & 0 & 1
\end{pmatrix}=
\begin{pmatrix}
c'K+x_0+z_0\\
-cK+y_0\\
z_0
\end{pmatrix},
$$
and, on the other side,
$$
\pi(A_2)+A_2\pi(B_2)=
\begin{pmatrix}
x_0\\
y_0\\
z_0
\end{pmatrix}+
\begin{pmatrix}
1 & 1 & 0\\
0 & 1 & 0\\
0 & 0 & 1
\end{pmatrix}
\begin{pmatrix}
c'K \\
-cK \\
0 
\end{pmatrix}=
\begin{pmatrix}
x_0+c'K-cK\\
y_0-cK\\
z_0
\end{pmatrix}.
$$
Looking at the first component, we see that $z_0$ must be equal to $-cK$. But
right above we have reasoned that $z_0$ has to be equal to $cK$,
 so we conclude that $z_0=0$, a contradiction.

When $p=2$, we have $K=z_0=b'=c=1$, $bc'=0$ and $y_0=0$ (if $y_0\neq 0$, the element $(x_0,y_0,z_0)$ has order $4$,
a contradiction with the matrix of order $2$ that we have assigned to it), so $\pi$ is simply 
$$
\pi\begin{pmatrix}
1 & n & x\\
0 & 1 & y\\
0 & 0 & 1
\end{pmatrix}=
\begin{pmatrix}
c'x+y+x_0n+xn\\
by+x+ny\\
n
\end{pmatrix}.
$$
We will see that there are four cases in which $\pi$ gives rise to a structure of left brace with trivial 
socle, and that all four of them 
give rise to isomorphic braces. To prove that $\pi$ is bijective, we must check that
$
\pi\begin{pmatrix}
1 & n & x\\
0 & 1 & y\\
0 & 0 & 1
\end{pmatrix}=
\pi\begin{pmatrix}
1 & m & x'\\
0 & 1 & y'\\
0 & 0 & 1
\end{pmatrix}
$ implies $n=m$, $x=x'$ and $y=y'$. We have directly $n=m$ looking at the third component of $\pi$. 
As for the first component and the second component, 
$$
\left\lbrace\begin{array}{rcl}
c'x+y+x_0n+xn&=&c'x'+y'+x_0n+x'n,\\
by+x+ny&=&by'+x'+ny'.
\end{array}\right.
$$
For $\pi$ to be bijective, the only solution to this system of equations has to be $x=x'$ and $y=y'$.
Equivalently, the following linear system of equations in $X$ and $Y$ cannot have a solution different from zero 
$$
\left\lbrace
\begin{array}{lr}
c'X+Y+nX&=0,\\
bY+X+nY&=0.
\end{array}\right.
$$
In other words, the determinant of the matrix associated to the system 
$$
\begin{vmatrix}
1 & b+n\\
c'+n & 1
\end{vmatrix}=1+(b+n)(c'+n)=1+(b+c')n+n^2
$$
has to be different from zero. The last equation has no solution on $n$ 
if and only if $b+c'=1$, or, equivalently, if and only if $b\neq c'$.

Now, to check $\pi(AB)=\pi(A)+A\pi(B)$, take two matrices $A$ and $B$ in $T$, 
and suppose
$A=\begin{pmatrix}
1 & n & x\\
0 & 1 & y\\
0 & 0 & 1
\end{pmatrix}
$
and 
$B=\begin{pmatrix}
1 & m & x'\\
0 & 1 & y'\\
0 & 0 & 1
\end{pmatrix}
$. We have
$$
AB=\begin{pmatrix}
1 & n+m & x+x'+n y'\\
0 & 1 & y+y'\\
0 & 0 & 1
\end{pmatrix},
$$
and then,
$$
\pi(AB)=
\begin{pmatrix}
c'(x+x'+ny')+(y+y')+x_0(n+m)+(x+x'+ny')(n+m)\\
b(y+y')+(x+x'+ny')+(n+m)(y+y')\\
n+m
\end{pmatrix}.
$$ 
On the other side,
\begin{eqnarray*}
\pi(A)+A\pi(B)&=&
\begin{pmatrix}
c'x+y+x_0n+xn\\
by+x+ny\\
n
\end{pmatrix}+
\begin{pmatrix}
1 & n & x\\
0 & 1 & y\\
0 & 0 & 1
\end{pmatrix}
\begin{pmatrix}
c'x'+y'+x_0m+mx' \\
by'+x'+my' \\
m 
\end{pmatrix}\\[7pt]
&=&\begin{pmatrix}
c'x+y+x_0n+xn+c'x'+y'+x_0m+mx'+n(by'+x'+my')+xm\\
by+x+ny+by'+x'+my'+ym\\
n+m
\end{pmatrix}.
\end{eqnarray*}
It is immediate to check that the second and the third components of each expression are equal. For the first one, 
we obtain $y'n(1+c'+b)=0$. Since $c'+ b=1$, this is always true in this case.

With the conditions $c\neq 0$, $b'\neq 0$, $c'\neq b$, $c'b=0$, $y_0=0$ and $z_0\neq 0$, there are
four possible cases in which $\pi$ give rise to a structure of brace:
\begin{enumerate}
\item $x_0=1$, $b=1$, $c'=0$;
\item $x_0=1$, $b=0$, $c'=1$;
\item $x_0=0$, $b=1$, $c'=0$;
\item $x_0=0$, $b=0$, $c'=1$.
\end{enumerate}
The first one gives the following values for $\pi$:
$$
\begin{array}{cc}
(0,0,0)\mapsto\begin{pmatrix}
1 & 0 & 0\\
0 & 1 & 0\\
0 & 0 & 1
\end{pmatrix}, &
(1,0,0)\mapsto\begin{pmatrix}
1 & 0 & 1\\
0 & 1 & 1\\
0 & 0 & 1
\end{pmatrix},\vspace{7pt}\\
(0,1,0)\mapsto\begin{pmatrix}
1 & 0 & 1\\
0 & 1 & 0\\
0 & 0 & 1
\end{pmatrix},&
(0,0,1)\mapsto\begin{pmatrix}
1 & 1 & 0\\
0 & 1 & 1\\
0 & 0 & 1
\end{pmatrix},\vspace{7pt}\\
(1,1,0)\mapsto\begin{pmatrix}
1 & 0 & 0\\
0 & 1 & 1\\
0 & 0 & 1
\end{pmatrix}, &
(1,0,1)\mapsto\begin{pmatrix}
1 & 1 & 0\\
0 & 1 & 0\\
0 & 0 & 1
\end{pmatrix},\vspace{7pt}\\
(0,1,1)\mapsto\begin{pmatrix}
1 & 1 & 1\\
0 & 1 & 0\\
0 & 0 & 1
\end{pmatrix}, &
(1,1,1)\mapsto\begin{pmatrix}
1 & 1 & 1\\
0 & 1 & 1\\
0 & 0 & 1
\end{pmatrix}.\\
\end{array}
$$
The other cases are isomorphic to this one by the morphisms $F_i:G_1\to G_i$, where
\begin{enumerate}[1.]
\setcounter{enumi}{1}
\item $
F_2$ is equal to $\begin{pmatrix}
1 & 1 & 0\\
0 & 1 & 1\\
0 & 0 & 1
\end{pmatrix}$ as a morphism of the additives groups, and equal to  
the conjugation by $\begin{pmatrix}
1 & 1 & 1\\
0 & 1 & 1\\
0 & 0 & 1
\end{pmatrix}
$ as a morphism of the multiplicative groups;
\item $
F_3$ is equal to $\begin{pmatrix}
1 & 0 & 1\\
0 & 1 & 0\\
0 & 0 & 1
\end{pmatrix}$ as a morphism of the additive groups, and equal to 
the conjugation by $\begin{pmatrix}
1 & 1 & 1\\
0 & 1 & 1\\
0 & 0 & 1
\end{pmatrix}$ as a morphism of the multiplicative groups;
\item $
F_4$ is equal to $\begin{pmatrix}
1 & 1 & 1\\
0 & 1 & 1\\
0 & 0 & 1
\end{pmatrix}$ as a morphism of the additive groups, and equal to
the identity as a morphism of the multiplicative groups.
\end{enumerate}

\section{$(G,+)$ isomorphic to $\Z/(p)\times \Z/(p^2)$}
The strategy is the same as that of the previous section. Observe that
 we look for monomorphisms
$\sigma: G/\soc(G)\to \aut(\Z/(p)\times \Z/(p^2))$, and we first should find a
way to present the elements of $\aut(\Z/(p)\times \Z/(p^2))$ in a more convenient 
way to work with them.
$\aut(\Z/(p)\times \Z/(p^2))$ is isomorphic to 
$$M:=
\left\lbrace 
\begin{pmatrix}
x & y \\
p z & t 
\end{pmatrix}:
x,y\in\Z/(p),~ z,t\in\Z/(p^2),~xt\not\equiv 0\pmod{p}\right\rbrace,
$$
with multiplication
$$
\begin{pmatrix}
 x& y\\
 p z& t\\
\end{pmatrix}
\begin{pmatrix}
 a & b\\
 p c & d\\
\end{pmatrix}:=
\begin{pmatrix}
 xa\pmod{p} & xb+yd\pmod{p}\\
 p(za+tc)\pmod{p^2} & td+p zb\pmod{p^2}
\end{pmatrix},
$$
which is the usual multiplication of matrices over $\Z$, followed by a reduction modulo $p$
on the first row and a reduction modulo $p^2$ on the second one.

Observe also that one key step to reduce the computations in our previous cases
 is to move the image of $G/\soc(G)$ by $\sigma$ into
a Sylow $p$-subgroup, so we must compute a suitable Sylow $p$-subgroup of $M$. 
The order of $M$ is equal to
$\ord{M}=(p-1)p^2(p^2-p)=p^3(p-1)^2,$
so the subgroup
$$
M_p:=\left\lbrace 
\begin{pmatrix}
1 & c \\
p a & 1+p b 
\end{pmatrix}:
c\in\Z/(p),~a,b\in\Z/(p^2)\right\rbrace\leq M
$$
is a Sylow $p$-subgroup. Then, we may take $\sigma: G/\soc(G)\to M_p,$ because
$\sigma(G/\soc(G))$ is a $p$-subgroup, so it is contained
in a Sylow $p$-subgroup of $M$, any two 
Sylow $p$-subgroups are conjugate, and conjugate representations 
$\sigma$ give rise to isomorphic brace structures.

\begin{remark}
There is a problem when we work with matrices with entries in different rings, and 
we make computations in them. For instance, we can find an entry of $\Z/(p^2)$ of the form
$\alpha^{-1}$, where $\alpha$ is an element of $\Z/(p)$. To make this kind of expressions formally correct, 
during the computations we assume that the entries of our matrices are all over the $p$-adic numbers $\hat{\Z}_p$,
and at the end we reduce the first row modulo $p$ and the second row modulo $p^2$ to return to our usual group of matrices.
\end{remark}

\subsection{Socle of order $p^3$}
Since $G=\soc(G)$, the only possible structure is the trivial one.

\subsection{Socle of order $p^2$}
$G/\soc(G)$ has order $p$, so it is isomorphic to the trivial brace over $\Z/(p)$. 
The morphism $\sigma$ is determined by a matrix
$A=
\begin{pmatrix}
1 & c  \\
pa & 1+pb \\
\end{pmatrix}
$ in $M_p$ of order $p$, and the morphism $h$, by a non-zero vector $(\alpha,\beta)\in (\Z/(p))^2$, satisfying
$(\alpha,\beta) A=(\alpha,\beta)$. Then, the multiplication is given by
$$
\begin{pmatrix}
 x_1\\
 y_1\\
\end{pmatrix}\cdot
\begin{pmatrix}
 x_2\\
 y_2\\
\end{pmatrix}:=
\begin{pmatrix}
 x_1\\
 y_1\\
\end{pmatrix}
+A^{\alpha x_1+\beta y_1}
\begin{pmatrix}
x_2\\
y_2
\end{pmatrix}.
$$
The condition of isomorphism is 
$$
A^{(\alpha,~\beta)(x,y)^t}=(F^{-1}A'F)^{(\alpha',~\beta')F(x,y)^t},
$$
for some $F\in M.$

To reduce the possible multiplications, we have to see which elements of $M_p$ are conjugate by some element
of $M$. If 
$F=\begin{pmatrix}
X & Y  \\
pZ & T \\
\end{pmatrix}\in M,$ then
\begin{eqnarray*}
F^{-1}AF&=&
\begin{pmatrix}
X^{-1}& -YX^{-1}T^{-1}\\
-pZX^{-1}T^{-1}& T^{-1}+pZYX^{-1}T^{-2}\\
\end{pmatrix}
\begin{pmatrix}
1& c\\
pa& 1+pb\\
\end{pmatrix}
\begin{pmatrix}
X& Y\\
pZ& T\\
\end{pmatrix}\\[7pt]
&=&\begin{pmatrix}
1 & cTX^{-1}  \\
paT^{-1}X & 1+p(b+aT^{-1}Y-cX^{-1}Z) \\
\end{pmatrix}.
\end{eqnarray*}
Observe that, when $a=0$ and $c=0$, the only matrix conjugate to $A$ is $A$ itself. In the other cases, 
we can use the values of $Y$ or of $Z$ to turn $b$ into 0, and the values of $X$ and $T$ to turn 
$a$ or $c$ into 1 (but not both at the same time). For that reason, any element in $M_p$ is conjugate 
to one of the following three matrices:
\begin{enumerate}[(a)]
\item $
\begin{pmatrix}
1 & 0  \\
0 & 1+pb \\
\end{pmatrix},
$ $\forall b\neq 0$;
\item $
\begin{pmatrix}
1 & 0  \\
p & 1 \\
\end{pmatrix}
$; 
\item $
\begin{pmatrix}
1 & 1  \\
pa & 1 \\
\end{pmatrix},
$ $\forall a\in\{0,1,\dots,p-1\}$.
\end{enumerate}

We must determine the possible values for the vector $(\alpha,\beta)$ in each case. 
In the cases (a) and (b), any vector satisfies the condition $(\alpha,\beta) A=(\alpha,\beta)$. 
For the case (c), 
the only elements satisfying this condition are those of the form 
$(\alpha,\beta)=(0,k)$, $k\neq 0$, but using $F=k^{-1}\Id$, 
we can take $(\alpha,\beta)=(0,1)$.

It is also important to observe the effect of $F$ over $(\alpha,\beta)$
$$
(\alpha,\beta)F=(\alpha X,\alpha Y+\beta T).
$$
This allows us to simplify the vector $(\alpha,\beta)$ a bit more in cases (a) and (b). Then there are five possible cases:
\begin{enumerate}
\item $\begin{pmatrix}
 x_1\\
 y_1\\
\end{pmatrix}\cdot
\begin{pmatrix}
 x_2\\
 y_2\\
\end{pmatrix}:=
\begin{pmatrix}
 x_1\\
 y_1\\
\end{pmatrix}+
\begin{pmatrix}
1 & 0  \\
0 & 1+pb \\
\end{pmatrix}^{x_1}\begin{pmatrix}
x_2 \\
y_2
\end{pmatrix},$

\item $\begin{pmatrix}
 x_1\\
 y_1\\
\end{pmatrix}\cdot
\begin{pmatrix}
 x_2\\
 y_2\\
\end{pmatrix}:=
\begin{pmatrix}
 x_1\\
 y_1\\
\end{pmatrix}
+\begin{pmatrix}
1 & 0  \\
0 & 1+pb \\
\end{pmatrix}^{y_1}\begin{pmatrix}
x_2 \\
y_2
\end{pmatrix},$

\item $\begin{pmatrix}
 x_1\\
 y_1\\
\end{pmatrix}\cdot
\begin{pmatrix}
 x_2\\
 y_2\\
\end{pmatrix}:=
\begin{pmatrix}
 x_1\\
 y_1\\
\end{pmatrix}
+\begin{pmatrix}
1 & 0 \\
p & 1 \\
\end{pmatrix}^{y_1}\begin{pmatrix}
x_2 \\
y_2
\end{pmatrix},$

\item $\begin{pmatrix}
 x_1\\
 y_1\\
\end{pmatrix}\cdot
\begin{pmatrix}
 x_2\\
 y_2\\
\end{pmatrix}:=
\begin{pmatrix}
 x_1\\
 y_1\\
\end{pmatrix}
+\begin{pmatrix}
1 & 0 \\
p & 1 \\
\end{pmatrix}^{x_1+ay_1}\begin{pmatrix}
x_2 \\
y_2
\end{pmatrix},$

\item $\begin{pmatrix}
 x_1\\
 y_1\\
\end{pmatrix}\cdot
\begin{pmatrix}
 x_2\\
 y_2\\
\end{pmatrix}:=
\begin{pmatrix}
 x_1\\
 y_1\\
\end{pmatrix}
+\begin{pmatrix}
1 & 1 \\
pa & 1 \\
\end{pmatrix}^{y_1}\begin{pmatrix}
x_2 \\
y_2
\end{pmatrix}.$

\end{enumerate}
Using the isomorphism condition, we can reduce the fourth and the fifth cases a bit more. 
In 4, if $a\not\equiv 0\pmod{p}$,
$
F=\begin{pmatrix}
a & 0 \\
0 &  a^{2}\\
\end{pmatrix}
$ satisfies
$$
F\begin{pmatrix}1& 0\\p& 1\end{pmatrix}^{x+ay}=\begin{pmatrix}1& 0\\p& 1\end{pmatrix}^{(1,1)F(x,y)^t}F,
$$ 
so we may take $a=1$. In 5, if $a\not\equiv 0\pmod{p}$ and $p\neq 2$, there are two possibilities:
if $a$ is a square modulo $p$, 
$
F=\begin{pmatrix}
a & 0 \\
p(a^{1/2}-a)/2 &  a^{1/2}\\
\end{pmatrix}
$ 
satisfies 
$$
F\begin{pmatrix}1& 1\\pa& 1\end{pmatrix}^y=\begin{pmatrix}1& 1\\p& 1\end{pmatrix}^{(0,1)F(x,y)^t}F,
$$
so we may take $a=1$. If $a$ is not a square modulo $p$, fix a non-square element $\varepsilon$
of $\Z/(p)$. Then, 
$$
F=\begin{pmatrix}
\varepsilon^{-1} a & 0 \\
pa ((\varepsilon^{-1} a)^{1/2}-1)/2 &  (\varepsilon^{-1} a)^{1/2}\\
\end{pmatrix}
$$ 
satisfies
$$
F\begin{pmatrix}1& 1\\pa& 1\end{pmatrix}^y=\begin{pmatrix}1& 1\\p\varepsilon& 1\end{pmatrix}^{(0,1)F(x,y)^t}F,
$$
so we may take $a=\varepsilon$.
In conclusion, this gives the non-isomorphic cases

\begin{enumerate}
\item $\begin{pmatrix}
 x_1\\
 y_1\\
\end{pmatrix}\cdot
\begin{pmatrix}
 x_2\\
 y_2\\
\end{pmatrix}:=
\begin{pmatrix}
 x_1\\
 y_1\\
\end{pmatrix}
+\begin{pmatrix}
1 & 0  \\
0 & 1+p \\
\end{pmatrix}^{x_1}\begin{pmatrix}
x_2 \\
y_2
\end{pmatrix},$

\item $\begin{pmatrix}
 x_1\\
 y_1\\
\end{pmatrix}\cdot
\begin{pmatrix}
 x_2\\
 y_2\\
\end{pmatrix}:=
\begin{pmatrix}
 x_1\\
 y_1\\
\end{pmatrix}
+\begin{pmatrix}
1 & 0  \\
0 & 1+p \\
\end{pmatrix}^{y_1}\begin{pmatrix}
x_2 \\
y_2
\end{pmatrix},$

\item $\begin{pmatrix}
 x_1\\
 y_1\\
\end{pmatrix}\cdot
\begin{pmatrix}
 x_2\\
 y_2\\
\end{pmatrix}:=
\begin{pmatrix}
 x_1\\
 y_1\\
\end{pmatrix}
+\begin{pmatrix}
1 & 0 \\
p & 1 \\
\end{pmatrix}^{y_1}\begin{pmatrix}
x_2 \\
y_2
\end{pmatrix},$

\item $\begin{pmatrix}
 x_1\\
 y_1\\
\end{pmatrix}\cdot
\begin{pmatrix}
 x_2\\
 y_2\\
\end{pmatrix}:=
\begin{pmatrix}
 x_1\\
 y_1\\
\end{pmatrix}
+\begin{pmatrix}
1 & 0 \\
p & 1 \\
\end{pmatrix}^{x_1}\begin{pmatrix}
x_2 \\
y_2
\end{pmatrix},$

\item $\begin{pmatrix}
 x_1\\
 y_1\\
\end{pmatrix}\cdot
\begin{pmatrix}
 x_2\\
 y_2\\
\end{pmatrix}:=
\begin{pmatrix}
 x_1\\
 y_1\\
\end{pmatrix}
+\begin{pmatrix}
1 & 0 \\
p & 1 \\
\end{pmatrix}^{x_1+y_1}\begin{pmatrix}
x_2 \\
y_2
\end{pmatrix},$

\item $\begin{pmatrix}
 x_1\\
 y_1\\
\end{pmatrix}\cdot
\begin{pmatrix}
 x_2\\
 y_2\\
\end{pmatrix}:=
\begin{pmatrix}
 x_1\\
 y_1\\
\end{pmatrix}
+\begin{pmatrix}
1 & 1 \\
0 & 1 \\
\end{pmatrix}^{y_1}\begin{pmatrix}
x_2 \\
y_2
\end{pmatrix},$

\item $p\neq 2$, $\begin{pmatrix}
 x_1\\
 y_1\\
\end{pmatrix}\cdot
\begin{pmatrix}
 x_2\\
 y_2\\
\end{pmatrix}:=
\begin{pmatrix}
 x_1\\
 y_1\\
\end{pmatrix}
+\begin{pmatrix}
1 & 1 \\
p & 1 \\
\end{pmatrix}^{y_1}\begin{pmatrix}
x_2 \\
y_2
\end{pmatrix},$

\item $p\neq 2$, $\begin{pmatrix}
 x_1\\
 y_1\\
\end{pmatrix}\cdot
\begin{pmatrix}
 x_2\\
 y_2\\
\end{pmatrix}:=
\begin{pmatrix}
 x_1\\
 y_1\\
\end{pmatrix}
+\begin{pmatrix}
1 & 1 \\
p\varepsilon & 1 \\
\end{pmatrix}^{y_1}\begin{pmatrix}
x_2 \\
y_2
\end{pmatrix}.$
\end{enumerate}
The two last cases are not consider when $p=2$ because the matrix $A$ has order 4.

We compute now the multiplicative group of each case. When $p=2$,
\begin{enumerate}
 \item The multiplication is noncommutative, and $(1,0)$ and $(1,1)$ are two elements of 
 order $2$, so $(G,\cdot)\cong D_4$.
 \item The exponent is $2$, so $(G,\cdot)\cong (\Z/(2))^3$.
 \item The multiplication is noncommutative, and $(1,0)$ and $(1,1)$ are two elements of 
 order $2$, so $(G,\cdot)\cong D_4$.
 \item The multiplication is commutative and the exponent is $4$, so $(G,\cdot)\cong\Z/(2)\times\Z/(4)$.
 \item The multiplication is noncommutative, and $(0,2)$ is the unique element of 
 order $2$, so $(G,\cdot)\cong Q$.
 \item The multiplication is commutative and the exponent is $4$, so $(G,\cdot)\cong\Z/(2)\times\Z/(4)$.
\end{enumerate}
When $p\neq 2$, the exponent is always $p^2$, so
\begin{enumerate}
\item The multiplication is noncommutative, and $(G,\cdot)\cong M_3(p)$.
\item The multiplication is commutative, and $(G,\cdot)\cong\Z/(p)\times\Z/(p^2)$.
\item The multiplication is noncommutative, and $(G,\cdot)\cong M_3(p)$.
\item The multiplication is commutative, and $(G,\cdot)\cong\Z/(p)\times\Z/(p^2)$.
\item The multiplication is noncommutative, and $(G,\cdot)\cong M_3(p)$.
\item The multiplication is commutative, and $(G,\cdot)\cong\Z/(p)\times\Z/(p^2)$.
\item The multiplication is noncommutative, and $(G,\cdot)\cong M_3(p)$.
\item The multiplication is noncommutative, and $(G,\cdot)\cong M_3(p)$.
\end{enumerate}

\subsection{Socle of order $p$}

$G/\soc(G)$ is a brace of order $p^2$, and we have classified the possible structures 
of this type of braces.
But observe that in this case the additive group of $G/\soc(G)$ might be isomorphic to
 $\Z(p^2)$ or to $(\Z/(p))^2$, and thus we have to consider all the types of Proposition~\ref{bracesp2}.

\subsubsection{$G/\soc(G)$ is of type (iv)} We need two matrices $V$ and $W$ in $M_p$ of order $p$ such that $VW=WV$ and 
two linearly independent elements $h_1$ and $h_2$ of $\Z/(p)\times\Z/(p)$, satisfying $h_i W=h_i$ and $h_i V=h_i$ for $i=1,2$. 
Then, the multiplication is given by
$$
\begin{pmatrix}
 x_1\\
 y_1\\
\end{pmatrix}\cdot
\begin{pmatrix}
 x_2\\
 y_2\\
\end{pmatrix}:=
\begin{pmatrix}
 x_1\\
 y_1\\
\end{pmatrix}
+V^{h_1 (x_1,y_1)^t}W^{h_2 (x_1,y_1)^t}
\begin{pmatrix}
x_2\\
y_2
\end{pmatrix}.
$$
Two of these structures are isomorphic if 
$$
V^{h_1 (x,y)^t}W^{h_2 (x,y)^t}=(F^{-1}V'F)^{h'_1 F (x,y)^t}(F^{-1}W'F)^{h'_2 F (x,y)^t},
$$
for some $F\in M.$

By the conditions on $h_1$ and $h_2$, we must have 
$V=\begin{pmatrix}
1 & 0  \\
pa & 1+pb \\
\end{pmatrix}$
and
$W=\begin{pmatrix}
1 & 0  \\
pc & 1+pd \\
\end{pmatrix}.$
Then, if $h_1=(\alpha,\beta)$ and $h_2=(\gamma,\delta)$, we can make a rearrangement of the multiplication
\begin{eqnarray*}
\begin{pmatrix}
 x_1\\
 y_1\\
\end{pmatrix}\cdot
\begin{pmatrix}
 x_2\\
 y_2\\
\end{pmatrix}&:=&
\begin{pmatrix}
 x_1\\
 y_1\\
\end{pmatrix}+
\begin{pmatrix}
1 & 0  \\
pa & 1+pb \\
\end{pmatrix}^{\alpha x_1+\beta y_1}
\begin{pmatrix}
1 & 0  \\
pc & 1+pd \\
\end{pmatrix}^{\gamma x_1+\delta y_1}
\begin{pmatrix}
x_2\\
y_2\\
\end{pmatrix}\\[7pt]
&=&\begin{pmatrix}
  x_1\\
  y_1\\
 \end{pmatrix}
+
\begin{pmatrix}
1 & 0  \\
p(a\alpha+c\gamma)x_1+p(a\beta+c\delta)y_1 & 1+p(b\alpha+d\gamma)x_1+p(b\beta+d\delta)y_1 \\
\end{pmatrix}
\begin{pmatrix}
x_2\\
y_2\\
\end{pmatrix}\\[7pt]
&=&\begin{pmatrix}
  x_1\\
  y_1\\
 \end{pmatrix}
+
\begin{pmatrix}
1 & 0  \\
p & 1 \\
\end{pmatrix}^{Ax_1+Cy_1}
\begin{pmatrix}
1 & 0  \\
0 & 1+p \\
\end{pmatrix}^{Bx_1+Dy_1}
\begin{pmatrix}
x_2\\
y_2\\
\end{pmatrix},
\end{eqnarray*}
where $A=a\alpha+c\gamma$, $C=a\beta+c\delta$, $B=b\alpha+d\gamma$ and $D=b\beta+d\delta$.
The vectors $(A,C)$ and $(B,D)$ must be linearly independent for the brace to have socle
of order $p$. After this rearrangement, the isomorphism condition becomes
$$
F\begin{pmatrix}
1 & 0  \\
p & 1 \\
\end{pmatrix}^{Ax+Cy}
\begin{pmatrix}
1 & 0  \\
0 & 1+p \\
\end{pmatrix}^{Bx+Dy}=
\begin{pmatrix}
1 & 0  \\
p & 1 \\
\end{pmatrix}^{(A',C')F(x,y)^t}
\begin{pmatrix}
1 & 0  \\
0 & 1+p \\
\end{pmatrix}^{(B',D')F(x,y)^t}F.
$$
If $F=
\begin{pmatrix}
X & Y\\
pZ & T
\end{pmatrix},
$
then a direct computation of the left-hand side shows that
$$
F\begin{pmatrix}
1 & 0  \\
p & 1 \\
\end{pmatrix}^{Ax+Cy}
\begin{pmatrix}
1 & 0  \\
0 & 1+p \\
\end{pmatrix}^{Bx+Dy}=
\begin{pmatrix}
X & Y  \\
pZ+pT(A,C)(x,y)^t & T+pT(B,D)(x,y)^t \\
\end{pmatrix}.
$$
And the same type of computation for the right-hand side gives
\begin{eqnarray*}
\lefteqn{\begin{pmatrix}
1 & 0  \\
p & 1 \\
\end{pmatrix}^{(A',C')F(x,y)^t}
\begin{pmatrix}
1 & 0  \\
0 & 1+p \\
\end{pmatrix}^{(B',D')F(x,y)^t}F}\\[7pt]
&=&\begin{pmatrix}
X & Y  \\
pZ+pX(A',C')F(x,y)^t & T+pY(A',C')F(x,y)^t+pT(B',D')F(x,y)^t \\
\end{pmatrix}.
\end{eqnarray*}

Equating each entry of the two matrices, we obtain
$$
T\begin{pmatrix}
A & C  \\
B & D \\
\end{pmatrix}=
\begin{pmatrix}
X & 0  \\
Y & T \\
\end{pmatrix}
\begin{pmatrix}
A' & C'  \\
B' & D' \\
\end{pmatrix}
\begin{pmatrix}
X & Y  \\
0 & T \\
\end{pmatrix},
\text{ over }\Z/(p).
$$

To see how this condition works, we will compute the effect of the elementary matrices,
which generates all the other matrices. For all $\mu\in\Z/(p)\setminus\{0\}$,
\begin{enumerate}[(i)]
\item When $
F=\begin{pmatrix}
\mu & 0 \\
0 & 1 \\
\end{pmatrix},
~~
1^{-1}
\begin{pmatrix}
\mu & 0 \\
0 & 1 \\
\end{pmatrix}
\begin{pmatrix}
A & C \\
B & D \\
\end{pmatrix}
\begin{pmatrix}
\mu & 0 \\
0 & 1 \\
\end{pmatrix}
=
\begin{pmatrix}
\mu^2 A & \mu C \\
\mu B & D \\
\end{pmatrix}.
$

\item When $
F=\begin{pmatrix}
1 & 0 \\
0 & \mu \\
\end{pmatrix},~~
\mu^{-1}
\begin{pmatrix}
1 & 0 \\
0 & \mu \\
\end{pmatrix}
\begin{pmatrix}
A & C \\
B & D \\
\end{pmatrix} 
\begin{pmatrix}
1 & 0 \\
0 & \mu \\
\end{pmatrix}=
\begin{pmatrix}
\mu^{-1}A & C \\
B & \mu D \\
\end{pmatrix}.
$

\item When $
F=\begin{pmatrix}
1 & \mu \\
0 & 1 \\
\end{pmatrix},~~1^{-1}\begin{pmatrix}
1 & 0 \\
\mu & 1 \\
\end{pmatrix}
\begin{pmatrix}
A & C \\
B & D \\
\end{pmatrix}
\begin{pmatrix}
1 & \mu \\
0 & 1 \\
\end{pmatrix}=
\begin{pmatrix}
A & C+\mu A \\
B+\mu A & D+\mu(B+C)+\mu^2 A \\
\end{pmatrix}.
$
\end{enumerate}

When $A=0$, we can use a change of type (i) to turn $B$ into 1. Then, if $D=0$ we are done, and if $D\neq 0$, there 
are two possibilities: if $C\neq-1$, use a change of type (iii) to turn $D$ into 0; if $C=-1$, 
if $D=0$ we are done, if not, use a change of type (ii) to turn $D$ into 1. In all other cases, 
we can use a change of type (iii) to turn $B$ into 0. Then, using changes of type (ii), turn $D$ into 1,
and using changes of type (i), turn $A$ into 1 or $\varepsilon$ depending if $A$ is either a square or not.
In this last case, you can turn the value of $C$ into $-C$ using a change of type (i) with $\mu=-1$.

When $p=2$, there are the following non-isomorphic cases:
\begin{enumerate}
\item $
\begin{pmatrix}
A & C  \\
B & D \\
\end{pmatrix}=
\begin{pmatrix}
1 & 0  \\
0 & 1 \\
\end{pmatrix},
$
\item $
\begin{pmatrix}
A & C  \\
B & D \\
\end{pmatrix}=
\begin{pmatrix}
1 & 1  \\
0 & 1 \\
\end{pmatrix},
$
\item $
\begin{pmatrix}
A & C  \\
B & D \\
\end{pmatrix}=
\begin{pmatrix}
0 & 1  \\
1 & 0 \\
\end{pmatrix},
$
\item $
\begin{pmatrix}
A & C  \\
B & D \\
\end{pmatrix}=
\begin{pmatrix}
0 & 1  \\
1 & 1 \\
\end{pmatrix}.
$
\end{enumerate}

When $p\neq 2$, fix an element $\varepsilon$ which is not a square in $\Z/(p)$. 
Then, there are the following non-isomorphic cases:
\begin{enumerate}
\item
$
\begin{pmatrix}
A & C  \\
B & D \\
\end{pmatrix}=
\begin{pmatrix}
\varepsilon & \lambda  \\
0 & 1 \\
\end{pmatrix},
$ for each $\lambda\in\{0,1,\dots,(p-1)/2\}$;
\item $
\begin{pmatrix}
A & C  \\
B & D \\
\end{pmatrix}=
\begin{pmatrix}
1 & \lambda  \\
0 & 1 \\
\end{pmatrix},
$ for each $\lambda\in\{0,1,\dots,(p-1)/2\}$;
\item $
\begin{pmatrix}
A & C  \\
B & D \\
\end{pmatrix}=
\begin{pmatrix}
0 & \lambda  \\
1 & 0 \\
\end{pmatrix},
$ for each $\lambda\in\Z/(p)\setminus\{0\}$;
\item $
\begin{pmatrix}
A & C  \\
B & D \\
\end{pmatrix}=
\begin{pmatrix}
0 & -1  \\
1 & 1 \\
\end{pmatrix}.
$
\end{enumerate}
To check that they give rise to non-isomorphic braces, we compute a general change matrix.
\begin{enumerate}
\item
$
T^{-1}\begin{pmatrix}
X & 0  \\
Y & T \\
\end{pmatrix}
\begin{pmatrix}
\varepsilon & \lambda  \\
0 & 1 \\
\end{pmatrix}
\begin{pmatrix}
X & Y  \\
0 & T \\
\end{pmatrix}=
\begin{pmatrix}
\displaystyle\frac{X^2}{T}\varepsilon & \displaystyle\frac{XY}{T}\varepsilon+\lambda X\\
\displaystyle\frac{XY}{T}\varepsilon & \displaystyle\frac{Y^2}{T}\varepsilon+\lambda Y+T
\end{pmatrix}
$,\vspace{5pt}
\item $
T^{-1}\begin{pmatrix}
X & 0  \\
Y & T \\
\end{pmatrix}
\begin{pmatrix}
1 & \lambda  \\
0 & 1 \\
\end{pmatrix}
\begin{pmatrix}
X & Y  \\
0 & T \\
\end{pmatrix}=
\begin{pmatrix}
\displaystyle\frac{X^2}{T} & \displaystyle\frac{XY}{T}+\lambda X\\
\displaystyle\frac{XY}{T} & \displaystyle\frac{Y^2}{T}+\lambda Y+T
\end{pmatrix}
$, \vspace{5pt}
\item $
T^{-1}\begin{pmatrix}
X & 0  \\
Y & T \\
\end{pmatrix}
\begin{pmatrix}
0 & \lambda  \\
1 & 0 \\
\end{pmatrix}
\begin{pmatrix}
X & Y  \\
0 & T \\
\end{pmatrix}=
\begin{pmatrix}
0 & \lambda X\\
X & Y(1+\lambda)
\end{pmatrix}
$,\vspace{5pt}
\item $
T^{-1}\begin{pmatrix}
X & 0  \\
Y & T \\
\end{pmatrix}
\begin{pmatrix}
0 & -1  \\
1 & 1 \\
\end{pmatrix}
\begin{pmatrix}
X & Y  \\
0 & T \\
\end{pmatrix}=
\begin{pmatrix}
0 & -X\\
X & T
\end{pmatrix}.
$
\end{enumerate}

Cases 3 and 4 are not isomorphic to cases 1 and 2 because their $(1,1)$-entry is always equal to zero.
Case 3 is not isomorphic to case 4 because if 
$\begin{pmatrix}
0 & -X\\
X & T
\end{pmatrix}
$ was equal to 
$\begin{pmatrix}
0 & \lambda  \\
1 & 0 \\
\end{pmatrix}$, that would imply $T=0$, which is impossible.
Case 1 is not isomorphic to case 2 because, if $
\begin{pmatrix}
\displaystyle\frac{X^2}{T} & \displaystyle\frac{XY}{T}+\lambda X\\
\displaystyle\frac{XY}{T} & \displaystyle\frac{Y^2}{T}+\lambda Y+T
\end{pmatrix}
$ was equal to 
$\begin{pmatrix}
\varepsilon & \lambda  \\
0 & 1 \\
\end{pmatrix}$, then $Y$ would be 0, $T$ would be 1, and $X^2$ would be equal to $\varepsilon$, a
contradiction with the fact that $\varepsilon$ is not a square.
Finally, different values of $\lambda$ in case 1 determine non-isomorphic braces because
$\begin{pmatrix}
1 & \lambda'  \\
0 & 1 \\
\end{pmatrix}=\begin{pmatrix}
\displaystyle\frac{X^2}{T} & \displaystyle\frac{XY}{T}+\lambda X\\
\displaystyle\frac{XY}{T} & \displaystyle\frac{Y^2}{T}+\lambda Y+T
\end{pmatrix}
$ 
implies $Y=0$, $T=1$, $\lambda'=\lambda X$ and $X^2=1$, so $\lambda'=\pm\lambda$, and if 
$\lambda, \lambda'\in\{0,1,\dots,(p-1)/2\}$, then $\lambda'=\lambda$. Analogously, in case 2, different 
values of $\lambda$ determine non-isomorphic braces.

We compute now the multiplicative group of each case. If $p=2$,
\begin{enumerate}
 \item It is abelian and the exponent is $4$, so $(G,\cdot)\cong \Z/(2)\times\Z/(4)$.
 \item It is non-abelian, and there are more than one element of order $2$ 
 (for instance, $(0,2)$ and $(1,0)$), so $(G,\cdot)\cong D_4$.
 \item It is abelian and the exponent is $4$, so $(G,\cdot)\cong \Z/(2)\times\Z/(4)$.
 \item The exponent is $2$, so $(G,\cdot)\cong (\Z/(2))^3$.
\end{enumerate}
If $p\neq 2$, the exponent is always $p^2$, and then
\begin{enumerate}
 \item The multiplication is commutative if and only if $\lambda=0$. Thus $(G,\cdot)\cong \Z/(p)\times\Z/(p^2)$
 when $\lambda=0$, and $(G,\cdot)\cong M_3(p)$ when $\lambda\neq 0$.
 \item The multiplication is commutative if and only if $\lambda=0$. Thus $(G,\cdot)\cong \Z/(p)\times\Z/(p^2)$
 when $\lambda=0$, and $(G,\cdot)\cong M_3(p)$ when $\lambda\neq 0$.
 \item The multiplication is commutative if and only if $\lambda=1$. Thus $(G,\cdot)\cong \Z/(p)\times\Z/(p^2)$
 when $\lambda=1$, and $(G,\cdot)\cong M_3(p)$ when $\lambda\neq 1$.
 \item The multiplication is noncommutative, so $(G,\cdot)\cong M_3(p)$.
\end{enumerate}

\subsubsection{$G/\soc(G)$ is of type (v)}

\paragraph{Case $p\neq 2$.} In this case, $\sigma$ is determined by two matrices $A$ and $B$ of order $p$ 
such that $AB=BA$ in the following way
$$
\begin{array}{cccc}
\sigma:& (G/\soc(G),\cdot)\cong (\Z/(p))^2&\longrightarrow &M\\
		&(1,0) &\mapsto &A\\
		&(0,1) &\mapsto &B\\
		&(x,y)=(1,0)^{x-C(y, 2)}(0,1)^y &\mapsto &A^{x-C(y,2)}B^y.
\end{array}
$$
On the other hand, $h$ is determined by two linearly independent vectors $h_1$ and 
$h_2$ of $(\Z/(p))^2$ such that $h_1 A=h_1$, $h_2 A=h_2$, $h_1 B=h_1+h_2$, $h_2 B=h_2$. 
Then, the multiplication is given by
$$
\begin{pmatrix}
 x_1\\
 y_1\\
\end{pmatrix}\cdot
\begin{pmatrix}
 x_2\\
 y_2\\
\end{pmatrix}:=
\begin{pmatrix}
 x_1\\
 y_1\\
\end{pmatrix}
+A^{ h_1 (x_1,y_1)^t-C(h_2 (x_1,y_1)^t,~ 2)}B^{h_2 (x_1,y_1)^t}
\begin{pmatrix}
x_2\\
y_2
\end{pmatrix}.
$$
Two of these structures are isomorphic if 
$$
A^{h_1 (x,y)^t-C(h_2 (x,y)^t,~ 2)}B^{h_2 (x,y)^t}=
(F^{-1}A'F)^{h'_1 F (x,y)^t-C(h'_2 F (x,y)^t,~ 2)}(F^{-1}B'F)^{h'_2 F (x,y)^t},
$$
for some $F\in M.$

As usual,
we can take $A$ and $B$ in $M_p$. Moreover, by the conditions on $h_1$ and $h_2$, we must have 
$A=\begin{pmatrix}
1 & 0  \\
pa & 1+pb \\
\end{pmatrix}$
and
$B=\begin{pmatrix}
1 & c'  \\
pa' & 1+pb' \\
\end{pmatrix},$
$c'\neq 0$, and $h_1=(\alpha,\beta)$, $h_2=(0,c'\alpha)$. But these two matrices only commute when $a=0$.
Using $F=\begin{pmatrix} 
c' & 0\\
p b' & c'
\end{pmatrix}$, we can take $b'=0$, and using then 
$F=\begin{pmatrix}
c' & 0\\
0 & 1
\end{pmatrix}
$, we can take $c'=1$. Then, using 
$F=\begin{pmatrix}
\alpha^{-1} & -\beta\alpha^{-2}\\
-pa'\beta\alpha^{-2} & \alpha^{-1}
\end{pmatrix}$, we can take $\alpha=1$, $\beta=0$.

So $A=\begin{pmatrix}
1 & 0  \\
0 & 1+pb \\
\end{pmatrix}$
and
$B=\begin{pmatrix}
1 & 1  \\
pa & 1 \\
\end{pmatrix},$
 and $h_1=(1,0)$, $h_2=(0,1).$ Hence
$$
\begin{pmatrix}
 x_1\\
 y_1\\
\end{pmatrix}\cdot
\begin{pmatrix}
 x_2\\
 y_2\\
\end{pmatrix}=
\begin{pmatrix}
 x_1\\
 y_1\\
\end{pmatrix}
+A^{x_1-C(y_1, 2)}B^{y_1}\begin{pmatrix}x_2\\y_2\end{pmatrix}
=\begin{pmatrix}
  x_1\\
  y_1\\
 \end{pmatrix}
+
\begin{pmatrix}
1 & y_1\\
pa y_1 & 1+pbx_1+p(a-b)C(y_1, 2)
\end{pmatrix}
\begin{pmatrix}x_2\\y_2\end{pmatrix}.
$$

Let's see if some of these braces are isomorphic. If 
$F=\begin{pmatrix}
X & Y\\
pZ & T
\end{pmatrix},
$
$A=\begin{pmatrix}
1 & 0\\
0 & 1+pb
\end{pmatrix},
$
$B=\begin{pmatrix}
1 & 1\\
pa & 1
\end{pmatrix},
$
$A'=\begin{pmatrix}
1 & 0\\
0 & 1+pb'
\end{pmatrix},
$
$B'=\begin{pmatrix}
1 & 1\\
pa' & 1
\end{pmatrix},
$
\vspace{5pt}
$$
FA^{x-C(y, 2)}B^{y}=
A'^{(1,0) F (x,y)^t-C((0,1) F (x,y)^t, 2)}B'^{(0,1) F (x,y)^t}F=
A'^{Xx+Yy-C(Ty, 2)}B'^{Ty}F,
$$

$$
FA^{x-C(y, 2)}B^{y}=
\begin{pmatrix}
X & Xy+Y\\
pZ+payT & T+p(a-b)T C(y, 2)+pbxT+pZy
\end{pmatrix},
$$

$$
A'^{Xx+Yy-C(Ty,2)}B'^{Ty}F=
\begin{pmatrix}
X & T^2y+Y\\
pZ+pa'TXy & T+pb'T(Xx+Yy-C(Ty, 2))+pa'TC(Ty, 2)+pa'TYy
\end{pmatrix}.
$$
This gives the relations 
$$
T^2=X,
$$
$$
a=a'X,
$$
$$
b=b'X,
$$
$$
(a'+b')TY+(b'-a')\frac{T^2}{2}=Z+(b-a)\frac{T}{2},
$$
$$
(a-b)=(a'-b')T^2.
$$

The fifth equation is a combination of the first three equations.
So taking $b'=1$ when $b$ is a square or $b'=\varepsilon$ when $b$ is not a square, and 
defining $Y=0$, $T=\sqrt{b/b'}$, $X=b/b'$, $Z=(b-a)\frac{1-T}{2}$,
we obtain that the non-isomorphic case are $b=1$ or $\varepsilon$, and any $a\in\{0,1,\dots,p-1\}$.
Thus the multiplications are given by
$$
\begin{pmatrix}
 x_1\\
 y_1\\
\end{pmatrix}\cdot
\begin{pmatrix}
 x_2\\
 y_2\\
\end{pmatrix}:=
\begin{pmatrix}
  x_1\\
  y_1\\
 \end{pmatrix}
+
\begin{pmatrix}
1 & y_1\\
pa y_1 & 1+px_1+p(a-1)C(y_1, 2)
\end{pmatrix}
\begin{pmatrix}x_2\\y_2\end{pmatrix},
$$
\vspace{4pt}
$$
\begin{pmatrix}
 x_1\\
 y_1\\
\end{pmatrix}\cdot
\begin{pmatrix}
 x_2\\
 y_2\\
\end{pmatrix}:=
\begin{pmatrix}
  x_1\\
  y_1\\
 \end{pmatrix}
+
\begin{pmatrix}
1 & y_1\\
pa y_1 & 1+p\varepsilon x_1+p(a-\varepsilon)C(y_1, 2)
\end{pmatrix}
\begin{pmatrix}x_2\\y_2\end{pmatrix}.
$$
Observe that the first multiplication is commutative if and only if $a=1$, and 
the second one is commutative if and only if $a=\varepsilon$.

When $p\neq 3$, the exponent of $(G,\cdot)$ 
is $p^2$ in both cases. In the first case, 
 $(G,\cdot)$ is isomorphic to 
$\Z/(p)\times\Z/(p^2)$ when $a=1$, and to $M_3(p)$ if $a\neq 1$.
In the second case,
$(G,\cdot)$ is isomorphic to 
$\Z/(p)\times\Z/(p^2)$ when $a=\varepsilon$, and to $M_3(p)$ if $a\neq \varepsilon$.

When $p=3$, we find different groups because the exponent is not always $p^2$. 
In the first case, $(G,\cdot)$ is isomorphic to $\Z/(3)\times\Z/(9)$ when $a=1$, to 
$M(3)$ when $a=-1$, and to $M_3(3)$ if $a=0$. In the second case, $\varepsilon$ must be equal to $-1$,
so $(G,\cdot)$ is isomorphic to $M_3(3)$ when $a=0$ or $1$, and to $(\Z/(3))^3$ if $a=-1$

\paragraph{Case $p=2$.} In this case, $(G/\soc(G),\cdot)\cong \Z/(4)$ and $(G/\soc(G),+)\cong \Z/(2)\times\Z/(2)$.
First of all, we need a matrix $A$ of order 4, that we may take in $M_p$.
The only matrices of order $4$ of this group are 
$\begin{pmatrix}
1 & 1 \\
2 & 1 \\
\end{pmatrix}$
and
$\begin{pmatrix}
1 & 1 \\
2 & 3 \\
\end{pmatrix}$, and these two matrices are conjugate by 
$F=\begin{pmatrix}
1 & 0 \\
2 & 1 \\
\end{pmatrix}$. So we can take 
$A=\begin{pmatrix}
1 & 1 \\
2 & 1 \\
\end{pmatrix}$.

On the other hand, we need two elements $h_1$ and $h_2$ of $\Z/(2)\times\Z/(2)$ such that $h_1 A=h_1+h_2$ and 
$h_2 A=h_2$. Therefore we have to take $h_2=(0,1)$ and $h_1=(1,0)$ or $(1,1)$. But using 
$F=\begin{pmatrix}
    1 & 1\\
    2 & 1\\
   \end{pmatrix}$, we may assume $h_1=(1,0)$.

In conclusion, there is only a brace with these characteristics up to isomorphism, with multiplication
$$
\begin{pmatrix}
 x_1\\
 y_1\\
\end{pmatrix}\cdot
\begin{pmatrix}
 x_2\\
 y_2\\
\end{pmatrix}:=
\begin{pmatrix}
 x_1\\
 y_1\\
\end{pmatrix}+
\begin{pmatrix}
1 & y_1 \\
2y_1 & 1+2x_1 \\
\end{pmatrix}
\begin{pmatrix}
x_2 \\
y_2 \\
\end{pmatrix}.
$$

We have $(G,\cdot)\cong \Z/(2)\times\Z/(4)$ because the multiplication is commutative, and there are elements
of order $4$, like $(0,1)$, but not of order $8$.

\subsubsection{$G/\soc(G)$ is of type (ii) and (iii)}

\paragraph{Case $p\neq 2$:} It cannot happen because there are no matrices of order $p^2$ in $M_p$.

\paragraph{Case $p=2$:} 

\paragraph{$G/Soc(G)$ is of type (ii).}
It cannot happen because there are no matrices $A$ of order 4 in $M_p$ and no elements $\alpha$, $\beta$ in $\Z/(4)$ 
such that $(\alpha,\beta)A=(\alpha,\beta)$ and $h(x,y)=\alpha x+\beta y$, $x\in\Z/(2)$, $y\in\Z/(4)$, 
is a surjective morphism to $\Z/(4).$

\paragraph{$G/Soc(G)$ is of type (iii).}
The bijective correspondence between the additive group and the multiplicative group of $G/\soc(G)$ is
$$
\pi: \Z/(4)\to\Z/(2)\times\Z/(2)
$$
$$
0\mapsto (0,0), 
$$
$$
1\mapsto (1,0), 
$$
$$
2\mapsto (0,1), 
$$
$$
3\mapsto (1,1),
$$
which can be written has
$$
\pi(z)=\left(z,~\sum_{i=1}^{z-1} i\right).
$$

Since $(G/\soc(G),\cdot)\cong \Z/(2)\times\Z/(2)$, we need two commuting matrices $A$ and $B$
of order 2. As before, we may take 
$A=\begin{pmatrix}
1 & c\\
2a & 1+2b
\end{pmatrix}$ and 
$B=\begin{pmatrix}
1 & c'\\
2a' & 1+2b'
\end{pmatrix}$. 
Also, we need a surjective morphism $h:\Z/(2)\times\Z/(4)\to\Z/(4)$. It is determined 
by $2\alpha=h(1,0)$ (which has to have order $2$) and $\beta=h(0,1)$, $2\alpha,\beta\in\Z/(4)$. 
The condition $h(\sigma(g)(m))=\lambda_g(h(m))$
is equivalent to 
$$(2\alpha,~\beta)A\begin{pmatrix}x\\y\end{pmatrix}=3(2\alpha,~\beta)\begin{pmatrix}x\\y\end{pmatrix}, \text{ and }$$ 
$$(2\alpha,~\beta)B\begin{pmatrix}x\\y\end{pmatrix}=(2\alpha,~\beta)\begin{pmatrix}x\\y\end{pmatrix},$$ 
which are equivalent to $a=a'=0$, $\alpha c+\beta b\equiv\beta\pmod{2}$, 
$\alpha c'+\beta b'\equiv 0\pmod{2}$. Using $\beta\equiv 1\pmod{2}$ and giving values to $\alpha$, we obtain four possible cases
$$
\alpha=0,~(b,c)=(1,0),~(b',c')=(0,1),~
A=\begin{pmatrix}
1& 0\\
0& 3\\
\end{pmatrix},~
B=\begin{pmatrix}
1& 1\\
0& 1\\
\end{pmatrix};
$$
$$
\alpha=0,~(b,c)=(1,1),~(b',c')=(0,1),~
A=\begin{pmatrix}
1& 1\\
0& 3\\
\end{pmatrix},~
B=\begin{pmatrix}
1& 1\\
0& 1\\
\end{pmatrix}; 
$$
$$
\alpha=1,~(b,c)=(0,1),~(b',c')=(1,1),~
A=\begin{pmatrix}
1& 1\\
0& 1\\
\end{pmatrix},~
B=\begin{pmatrix}
1& 1\\
0& 3\\
\end{pmatrix}; 
$$
$$
\alpha=1,~(b,c)=(1,0),~(b',c')=(1,1),~
A=\begin{pmatrix}
1& 0\\
0& 3\\
\end{pmatrix},~
B=\begin{pmatrix}
1& 1\\
0& 3\\
\end{pmatrix}. 
$$
First of all, we can turn $\beta$ into 1 without changing the matrices using 
$F=\begin{pmatrix}
1& 0\\
0& \beta^{-1}\\
\end{pmatrix}$. Then, the first and the fourth cases are conjugate, using 
$F=\begin{pmatrix}
1 & 0\\
2 & 1
\end{pmatrix}$. The second and the third cases are conjugate, using 
$F=\begin{pmatrix}
1 & 0\\
2 & 1
\end{pmatrix}$. And this two cases are isomorphic by 
$F=\begin{pmatrix}
1 & 0\\
0 & 3
\end{pmatrix}$ because
$$
A^{3y}B^{\sum_{i=1}^{3y-1} i}F=\begin{pmatrix}1& \sum^{3y-1}_{i=1} i\\ 0& 3+2y\end{pmatrix}=
\begin{pmatrix}1& \sum^{y}_{i=1} i\\ 0& 3+2y\end{pmatrix},
$$
$$
FA^yB^{\sum_{i=1}^{y-1} i}=\begin{pmatrix}1& y+\sum^{y-1}_{i=1} i\\ 0& 3+2y\end{pmatrix}.
$$

Thus, the only multiplication up to isomorphism is
\begin{eqnarray*}
\begin{pmatrix}
 x_1\\
 y_1\\
\end{pmatrix}\cdot
\begin{pmatrix}
 x_2\\
 y_2\\
\end{pmatrix}&=&
\begin{pmatrix}
 x_1\\
 y_1\\
\end{pmatrix}+
\sigma(h(x_1,y_1))(x_2,y_2)=
\begin{pmatrix}
 x_1\\
 y_1\\
\end{pmatrix}+
\sigma(y_1)(x_2,y_2)=
\begin{pmatrix}
 x_1\\
 y_1\\
\end{pmatrix}+
A^{y_1} B^{\sum_{i=1}^{y_1-1} i}
\begin{pmatrix}
x_2\\
y_2
\end{pmatrix}\\[7pt]
&=&
\begin{pmatrix}
 x_1\\
 y_1\\
\end{pmatrix}+
\begin{pmatrix}
1 & \displaystyle\sum_{i=1}^{y_1-1} i\\
0 & 1+2y_1
\end{pmatrix}
\begin{pmatrix}
x_2\\
y_2
\end{pmatrix}.
\end{eqnarray*}

In this case, $(G,\cdot)\cong D_4$ because the multiplication is noncommutative, and there are more than one element 
of order $2$ (for example, $(0,1)$ and $(1,0)$).

\subsection{Trivial socle}
As before, when the socle is trivial, the lambda map $\lambda:(G,\cdot)\to M_p$ becomes 
an isomorphism. Thus we are done if we could find a bijective map $\pi:M_p\to \Z/(p)\times \Z/(p^2)$ such that 
$\pi(AB)=\pi(A)+A\pi(B)$ for all $A,B\in M_p$

For $p\neq 2,3$, suppose that the matrix 
$A=\begin{pmatrix}
1 & c \\
p a & 1+p b 
\end{pmatrix}$ corresponds to the vector $v=\begin{pmatrix}x\\y\end{pmatrix}$, $x\in\Z/(p)$, $y\in\Z/(p^2)$.
Since $p\neq 2$, the matrix $A$ has order $p$, so $0=v^p=v+\lambda_v(v)+\lambda^2_v(v)+\cdots +\lambda^{p-1}_v(v)=
(\Id+A+A^2+\cdots +A^{p-1}) v$. Using induction, we obtain 
$A^n=\begin{pmatrix}
1 & nc \\
p na & 1+p nb+pC(n, 2) ac 
\end{pmatrix}$, and 
$\Id+A+A^2+\cdots +A^{p-1}=
\begin{pmatrix}
p & C(p, 2)c \\
p C(p, 2)a & p+pC(p, 2) b+ p\sum^{p-1}_{i=2}C(i, 2)ac
\end{pmatrix}=
\begin{pmatrix}
0 & 0 \\
0 & p
\end{pmatrix}
$
because $C(p, 2)\equiv 0\pmod{p}$ for $p\neq 2$ and $\sum^{p-1}_{i=2}C(i, 2)=p\frac{(p-1)(p-2)}{6}\equiv 0\pmod{p}$ 
if $p\neq 2,3$. In conclusion, any element $v$ must satisfy 
$0=\begin{pmatrix}
0 & 0 \\
0 & p
\end{pmatrix} v=
\begin{pmatrix}
0\\
p y
\end{pmatrix},
$
or, equivalently, $y\equiv 0\pmod{p},$ which is a contradiction.

For $p=3$, we obtain $\Id+A+\cdots+A^{p-1}=\Id+A+A^2=
\begin{pmatrix}
0 & 0 \\
0 & 3+3ac
\end{pmatrix}$. When $a=0$ or $c=0$, just like before, we need that $y\equiv 0\pmod{3}$. 
But there are 17 matrices with $a=0$ or $c=0$, and just 9 vectors with $y\equiv 0\pmod{3}$,
so we have less vectors than matrices to assign.

For $p=2$, any matrix $A$ satisfies 
$$
\Id+A=\begin{pmatrix}
0 & c\\
2a & 2+2b\\
\end{pmatrix},
$$
$$
\Id+A+A^2+A^3=0.
$$
Then, the conditions $(\Id+A)\pi(A)=0$ for matrices of order 2, and $(\Id+A)\pi(A)\neq 0$ for matrices of order 4,
 give the following necessary conditions:
\begin{itemize}
\item $(0,2)$ cannot be assigned to 
$\begin{pmatrix}
1 & 1\\
2 & 1\\
\end{pmatrix}$
nor
$\begin{pmatrix}
1 & 1\\
2 & 3\\
\end{pmatrix}$.
\item $\begin{pmatrix}
1 & 0\\
0 & 3\\
\end{pmatrix}$
has to be assigned to $(1,0)$, $(0,2)$ or $(1,2)$.
\item $\begin{pmatrix}
1 & 0\\
2 & 1\\
\end{pmatrix}$
has to be assigned to $(1,1)$, $(0,2)$ or $(1,3)$.
\item $\begin{pmatrix}
1 & 0\\
2 & 3\\
\end{pmatrix}$
has to be assigned to $(0,1)$, $(0,2)$ or $(0,3)$.
\item $\begin{pmatrix}
1 & 1\\
0 & 1\\
\end{pmatrix}$
has to be assigned to  $(1,0)$, $(0,2)$ or $(1,2)$.
\item $\pi\begin{pmatrix}
1 & 1\\
0 & 3\\
\end{pmatrix}=
\pi\left(\begin{pmatrix}
1 & 1\\
0 & 1
\end{pmatrix}
\begin{pmatrix}
1 & 0\\
0 & 3
\end{pmatrix}\right)=
\pi\begin{pmatrix}
1 & 1\\
0 & 1
\end{pmatrix}+
\begin{pmatrix}
1 & 1\\
0 & 1
\end{pmatrix}
\pi\begin{pmatrix}
1 & 0\\
0 & 3
\end{pmatrix}$ has to be equal to $(1,0)$, $(0,2)$ or $(1,2)$.
\end{itemize}
These conditions are summarize in the following diagram

$$
\left\lbrace(0,2),~(1,0),~(1,2)\right\rbrace\longleftrightarrow\left\lbrace
\begin{pmatrix}
1 & 1 \\
0 & 1
\end{pmatrix},~
\begin{pmatrix}
1 & 1 \\
0 & 3
\end{pmatrix},~
\begin{pmatrix}
1 & 0 \\
0 & 3
\end{pmatrix}\right\rbrace,
$$

$$
\left\lbrace(0,1),~(1,1),~(0,3),~(1,3)\right\rbrace\longleftrightarrow\left\lbrace
\begin{pmatrix}
1 & 1 \\
2 & 1
\end{pmatrix},~
\begin{pmatrix}
1 & 1 \\
2 & 3
\end{pmatrix},~
\begin{pmatrix}
1 & 0 \\
2 & 1
\end{pmatrix},~
\begin{pmatrix}
1 & 0 \\
2 & 3
\end{pmatrix}\right\rbrace,
$$
and $\begin{pmatrix}
1 & 0 \\
2 & 1
\end{pmatrix}$ goes to $(1,3)$ or $(1,1)$, and 
$\begin{pmatrix}
1 & 0 \\
2 & 3
\end{pmatrix}$,
 to $(0,3)$ or $(0,1)$.

When we assign a vector to $\begin{pmatrix}
1 & 1 \\
2 & 1
\end{pmatrix}
$, we have finished, because
$
\begin{pmatrix}
1 & 0 \\
0 & 3
\end{pmatrix}
$ and 
$
\begin{pmatrix}
1 & 1 \\
2 & 3
\end{pmatrix}
$
are powers of it, 
$
\begin{pmatrix}
1 & 0 \\
2 & 1
\end{pmatrix}
$
and
$
\begin{pmatrix}
1 & 0 \\
2 & 3
\end{pmatrix}
$
takes the elements of $\Z/(p)\times\Z/(p^2)$ that remains unassigned,  
and  $
\pi\begin{pmatrix}
1 & 1 \\
0 & 1
\end{pmatrix}$ and $
\pi\begin{pmatrix}
1 & 1 \\
0 & 3
\end{pmatrix}$ can be computed as

$$\pi\begin{pmatrix}
1 & 1 \\
0 & 1
\end{pmatrix}=
\pi\left(\begin{pmatrix}
1 & 1 \\
2 & 1
\end{pmatrix}
\begin{pmatrix}
1 & 0 \\
2 & 1
\end{pmatrix}\right)=
\pi\begin{pmatrix}
1 & 1 \\
2 & 1
\end{pmatrix}+
\begin{pmatrix}
1 & 1 \\
2 & 1
\end{pmatrix}
\pi\begin{pmatrix}
1 & 0 \\
2 & 1
\end{pmatrix},
$$
and
$$
\pi\begin{pmatrix}
1 & 1 \\
0 & 3
\end{pmatrix}=
\pi\left(\begin{pmatrix}
1 & 0 \\
2 & 1
\end{pmatrix}
\begin{pmatrix}
1 & 1 \\
2 & 1
\end{pmatrix}\right)=
\pi\begin{pmatrix}
1 & 0 \\
2 & 1
\end{pmatrix}+
\begin{pmatrix}
1 & 0 \\
2 & 1
\end{pmatrix}
\pi\begin{pmatrix}
1 & 1 \\
2 & 1
\end{pmatrix}.
$$

One of this possible assignations gives $\pi$ equal to
$$
\begin{array}{cc}
(0,0)\mapsto
\begin{pmatrix}
1 & 0 \\
0 & 1
\end{pmatrix}, &
(1,3)\mapsto
\begin{pmatrix}
1 & 0 \\
2 & 1
\end{pmatrix},
\vspace{7pt}\\
(0,1)\mapsto
\begin{pmatrix}
1 & 1 \\
2 & 1
\end{pmatrix}, &
(0,2)\mapsto
\begin{pmatrix}
1 & 1 \\
0 & 1
\end{pmatrix},
\vspace{7pt}\\
(1,2)\mapsto
\begin{pmatrix}
1 & 0 \\
0 & 3
\end{pmatrix}, &
(1,0)\mapsto
\begin{pmatrix}
1 & 1 \\
0 & 3
\end{pmatrix},
\vspace{7pt}\\
(1,1)\mapsto
\begin{pmatrix}
1 & 1 \\
2 & 3
\end{pmatrix}, &
(0,3)\mapsto
\begin{pmatrix}
1 & 0 \\
2 & 3
\end{pmatrix}.\\
\end{array}
$$
After some computations, we get
$$
\pi\left(
\begin{pmatrix}
1& c\\
2a& 1+2b
\end{pmatrix}\right)=
\begin{pmatrix}
a+b+c+ac\\
a+2\left(a+b+c+ab+\displaystyle\sum_{i=1}^{a-1} i\right)
\end{pmatrix},
$$
and then it is straightforward to check that $\pi(AB)=\pi(A)+A\pi(B)$ for all $A,B\in M_p$.

The other cases are isomorphic to this one by the morphisms $F_i:G_1\to G_i$, where
\begin{enumerate}[1.]
\setcounter{enumi}{1}
\item $G_2$ is obtained with the assignation 
$\pi\begin{pmatrix}
1 & 1 \\
2 & 1 
\end{pmatrix}=(1,1).$ $
F_2$ is equal to $\begin{pmatrix}
1 & 1 \\
2 & 1 
\end{pmatrix}$ as a morphism of the additive groups, and equal to
the conjugation by $\begin{pmatrix}
1 & 1 \\
2 & 1 
\end{pmatrix}
$ as a morphism of the multiplicative groups;
\item $G_3$ is obtained with the assignation 
$\pi\begin{pmatrix}
1 & 1 \\
2 & 1 
\end{pmatrix}=(0,3).$ $
F_3$ is equal to $\begin{pmatrix}
1 & 0 \\
0 & 3 
\end{pmatrix}$ as a morphism of the additive groups, and equal to
the identity as a morphism of the multiplicative groups;
\item $G_4$ is obtained with the assignation 
$\pi\begin{pmatrix}
1 & 1 \\
2 & 1 
\end{pmatrix}=(1,3).$
$
F_4$ is equal to $\begin{pmatrix}
1 & 1 \\
2 & 3 
\end{pmatrix}$ as a morphism of the additive groups, and equal to
the conjugation by $\begin{pmatrix}
1 & 1 \\
2 & 1 
\end{pmatrix}
$ as a morphism of the multiplicative groups.
\end{enumerate}

\section*{Acknowledgments}
 Research partially supported by DGI MINECO MTM2011-28992-C02-01, by
FEDER UNAB10-4E-378 ``Una ma\-ne\-ra de hacer Europa'', and by the Comissionat
per Universitats i Recerca de la Generalitat de Catalunya.

\vspace{30pt}
 \noindent \begin{tabular}{llllllll}
 D. Bachiller \\
 Departament de Matem\`atiques \\
 Universitat Aut\`onoma de Barcelona  \\
08193 Bellaterra (Barcelona), Spain  \\
dbachiller@mat.uab.cat 
\end{tabular}

\end{document}